\documentclass{article}
\usepackage[utf8]{inputenc}
\usepackage{amsmath}
\usepackage{amsthm}
\usepackage{amssymb}
\usepackage{mathtools}
\usepackage{xcolor}
\usepackage[hidelinks]{hyperref}
\usepackage[numbered]{bookmark}
\usepackage{enumerate}

\usepackage{cleveref}
  \crefname{section}{Section}{Sections}
  \crefname{figure}{Figure}{Figures}
  \crefname{theorem}{Theorem}{Theorems}
  \crefname{lemma}{Lemma}{Lemmas}
  \crefname{proposition}{Proposition}{Propositions}  
  \crefname{corollary}{Corollary}{Corollaries}
  \crefname{definition}{Definition}{Definitions}
  \crefname{example}{Example}{Examples}
  \crefname{remark}{Remark}{Remarks}

\newtheorem{theorem}{Theorem}[section]

\newtheorem{lemma}[theorem]{Lemma}

\newtheorem{proposition}[theorem]{Proposition}
\newtheorem{corollary}[theorem]{Corollary}

\newtheorem{problem}[theorem]{Problem}

\newtheorem{example}[theorem]{Example}

\newtheorem{remark}[theorem]{Remark}

\newcommand{\Teich}{Teichm\"{u}ller }
\newcommand{\Hol}{H\"{o}lder }

\newcommand{\tb}[1]{\textbf{#1}}
\newcommand{\bs}{\backslash}

\DeclareMathOperator{\diam}{diam}
\DeclareMathOperator{\fil}{fill}
\DeclareMathOperator{\hcap}{hcap}

\title{Drivers, hitting times, and weldings in Loewner's equation}
\author{Vlad Margarint and Tim Mesikepp}
\date{}

\begin{document}

\maketitle

\begin{abstract}
In addition to conformal weldings $\varphi$, simple curves $\gamma$ growing in the upper half plane generate driving functions $\xi$ and hitting times $\tau$ through Loewner's differential equation.  While the Loewner transform $\gamma \mapsto \xi$ and its inverse $\xi \mapsto \gamma$ have been carefully examined, less attention has been paid to the maps $\xi \mapsto \tau \mapsto \varphi$.  We study their continuity properties and show that uniform driver convergence implies uniform hitting time convergence and uniform welding convergence, even when the corresponding curves do not converge.  Welding convergence implies neither hitting time nor driver convergence, while hitting time convergence implies driver convergence in (at least) the case of constant drivers.

As an application, we show that a curve $\gamma$ of finite Loewner energy can be well approximated by an energy minimizer that matches $\gamma$'s welding on a sufficiently-fine mesh.
\end{abstract}

\section{Introduction and main results}

\subsection{Loewner's equation and associated functions}

A hundred years ago, Charles Loewner \cite{Loewner1923} showed the evolution of maps $g_t$ from the slit disk $\mathbb{D} \backslash \gamma([0,t])$ back to $\mathbb{D}$,  where $\gamma$ is a curve growing into $\mathbb{D}$ from its boundary, satisfy a differential equation which in effect transforms $\gamma$ into a continuous \emph{driving function} $\lambda(t) = g_t(\gamma(t))$ taking values on $\partial \mathbb{D}$. Loewner's approach played an important role in de Branges' proof \cite{bieberbach} of the Bierberbach conjecture in 1985, and received renewed interest following the ground-breaking 2000 work of Schramm \cite{Schramm2000}, who showed that the random curves generated by using a Brownian driving function run at speed $\kappa$ give the only-possible conformally-invariant scaling limits of a number of discrete models from statistical physics.  These processes, typically normalized to live in the upper half plane $\mathbb{H}$ now instead of Loewner's $\mathbb{D}$, and known as \emph{Schramm-Loewner-Evolutions SLE$_\kappa$}, have been intensely studied since and continue to be a topic of active research.\footnote{The literature is vast, and a non-exhaustive sampling of references is \cite{beff}, \cite{BLM}, \cite{JiamVlad},
\cite{Dubedat}, 
\cite{Jamespaper}, 
\cite{ShekharFriz},
\cite{FrizYuan}, 
\cite{GwynneMiller}, 
\cite{VikLawler}, 
\cite{lswerner},
\cite{LawlerVik2}, 
\cite{BasicpropSLE},
\cite{Schrammperc},
\cite{Qzipper}, 
\cite{Zhanrevers}.}  

In the setting of the upper half plane, Loewner's method takes a simple curve $\gamma:[0,T] \rightarrow \mathbb{H} \cup \{x\}$ and produces a real-valued driving function.  In the process, however, it also produces a \emph{hitting time} function and a \emph{conformal welding} $\varphi$.  To see this, consider the reversed Loewner flow in $\mathbb{H}$, described by normalized conformal maps $h_t:\mathbb{H} \rightarrow \mathbb{H}\bs \gamma_t$ which satisfy the ODE
\begin{align}\label{Eq:UpwardsLoewner}
    \partial_t h_t(z) = \frac{-2}{h_t(z) - \xi(t)}, \qquad h_0(z) = z
\end{align}
(see \S\ref{Sec:Prelim} for precise definitions and any unexplained terminology).  Here the $h_t$ conformally map $\mathbb{H}$ to the complement of the curve $\gamma_t$ (or more generally ``compact $\mathbb{H}$-hull'') generated by the continuous function $\xi:[0,T]\rightarrow \mathbb{R}$ on $[0,t]$. The dynamics in \eqref{Eq:UpwardsLoewner} extends to points $x \in \mathbb{R}$ away from $\xi(0)$, which by the ODE flow towards the driving function until the \emph{hitting time} 
\begin{align}\label{Def:HittingTime1}
    \tau(x) := \inf\{\,t \geq 0 \; : \; \liminf_{s \nearrow t} |h_s(x) - \xi(s)| =0  \,\}.
\end{align}
When $\xi$ is sufficiently regular, $\gamma_t$ neither intersects itself nor the real line, other than at its base $\xi(t)$, and in this case the function $x \mapsto \tau(x)$ is continuous.  Furthermore, exactly two points hit $\xi(t)$ to be ``welded together'' at each time $t$ (see Lemma \ref{Lemma:tauContinuous} below).  By \eqref{Eq:UpwardsLoewner} this pair started on opposite sides of $\xi(0)$, and we obtain another map, the \emph{conformal welding} $\varphi$, by sending $x < \xi(0)$ to the unique $\varphi(x) > \xi(0)$ which satisfies $\tau(x) = \tau(\varphi(x))$. We thus arrive at a natural sequence of mappings
\begin{align}\label{Intro:Maps}
    \gamma \mapsto \xi \mapsto \tau \mapsto \varphi.
\end{align}
The first arrow is the Loewner transform, and both it and its inverse have been carefully studied.\footnote{\label{Footnote:LoewnerTrans}Results include: $\gamma \mapsto \xi$ is continuous \cite[Thm. 6.2]{Kemp}, \cite[Thm. 1.8]{YizhengTopology}, but not uniformly so \cite[Figure 6]{LMR}.  $\xi \mapsto \gamma$ is not continuous with respect to capacity parametrization on the $\gamma$ \cite[Ex. 4.49]{Lawler}, but it is when the $\gamma$ are equipped with a Carath\'{e}odory-type topology on their normalized Riemann maps \cite[Prop. 4.47]{Lawler}.  In addition, restricting to more regular $\xi$ yields continuity for $\xi \mapsto \gamma$ with respect to various finer topologies on the $\gamma$.  See \cite[Thm. 4.1]{LMR} for $\xi$ with locally-small H\"{o}lder-$1/2$ norm, \cite[Thm. 2$(v)$,$(vi)$]{ShekharFriz} for $\xi$ of finite Loewner energy and \cite[Thm. 1.2]{ShekharWang} for $\xi$ of ``locally regular'' bounded variation.  See also \cite[Thm. 1.2]{SheffieldSun} for a continuity criterion in the $\xi \mapsto \gamma$ direction involving ``bi-directional and generic closeness'' on the drivers $\xi$.}  The latter maps have received less attention, however, and we preface and motivate our study of them by summarizing what is currently known.

Lind \cite{Lind4} used connections between $\xi, \tau$ and $\varphi$ as part of her argument that a driver $\xi$ with H\"{o}lder-1/2 seminorm $|\xi|_{1/2}<4$ generates a simple curve.  She proved that, given $|\xi|_{1/2}<4$, $\xi$ welds exactly two points at each time, $\tau(x) \asymp (x-\xi(0))^2$,  and $\varphi$ satisfies $(\varphi(x-h)-\varphi(x))/(\varphi(x)-\varphi(x+h)) \asymp 1$ \cite[Lemma 3, Corollary 1, Lemma 4]{Lind4}.  In our Lemma \ref{Lemma:tauContinuous} we generalize the first result to hold whenever $\xi$ generates a simple curve (which is known to be a broader class than $|\xi|_{1/2}<4$).

A connection between $\xi$ to $\tau$ appeared in Tran and Yuan's work on the topological support of SLE$_\kappa$ \cite{Yizheng}, where they proved that if two drivers $\xi$ and $\tilde{\xi}$ are $\delta$-close in sup norm, then $\tau(x) \leq T$ implies $\tilde{\tau}(x+\delta) \leq T$ \cite[Lemma 5.1]{Yizheng}, where $\tau$ and $\tilde{\tau}$ are the hitting times under the flows generated by $\xi$ and $\tilde{\xi}$, respectively.  We give a new simplified proof of this result in Lemma \ref{Lemma:DriverTimes1} and expand on the idea to show continuity properties of $\xi \mapsto \tau$ and $\tau \mapsto \varphi$.

Another recent work on SLE$_\kappa$ \cite{Vlad} studied the connection between $\xi$ and $\varphi$ in the case that $\xi$ a Brownian motion.  In it, the authors showed that the conformal welding $x \mapsto \varphi_\kappa(\omega, x)$ associated to SLE$_\kappa$ has a modification which is a.s. jointly continuous in $\kappa$ and $x$ for $(\kappa, x) \in [0,4]\times(-\infty,0]$.  One component of their proof was to show that, a.s., for all $0 \leq \kappa \leq 4$ simultaneously, the hitting times $x\mapsto \tau_\kappa(\omega, x)$ are continuous and strictly increasing on either side of $0=\xi_\kappa(0)$ \cite[Prop. 4.1$(a)$]{Vlad}.  This actually served as inspiration for the present study, as we wondered what could be said \emph{always} using deterministic Loewner theory, and not just with probability one. Our Lemma \ref{Lemma:tauContinuous} shows that $\tau$ always has these properties when $\xi$ generates a simple curve.  

Furthermore, note that the closeness of $\varphi_\kappa(\omega,x)$ and $\varphi_{\kappa'}(\omega,x)$ is a type of continuity statement about the driver-to-welding map $\xi \mapsto \varphi$, since when $\kappa'$ is sufficiently close to $\kappa$, the Brownian drivers $\xi_{\kappa'}(\omega,t) = \sqrt{\kappa'}B(\omega,t)$ and $\xi_{\kappa}(\omega,t) = \sqrt{\kappa}B(\omega,t)$ are close on a finite interval $[0,T]$ (note $\omega$ is fixed).  We generalize this joint continuity to hold for any drivers producing simple curves in Lemma \ref{Lemma:DriverToWeldingJoint}.  The authors also show \cite[Prop. 4.1$(b)$]{Vlad} that $\kappa \mapsto \tau_\kappa(\omega, \cdot)$ is continuous as a map from the reals to the space of continuous functions, which is likewise a type of continuity statement about $\xi \mapsto \tau$.  We extend this to all $\xi$ generating simple curves in Theorem \ref{Cor:TimesPointwise}.

This latter type of continuity question, for maps of functions to functions, is where our main interest lies, and this is what we study for $\xi \mapsto \varphi$, $\xi \mapsto \tau$ and $\tau \mapsto \varphi$.  

\subsection{Main results}

We equip spaces of continuous functions with the uniform norm, and restrict to the class of $\xi, \tau$, and $\varphi$ corresponding to simple curves $\gamma$.  This is natural to obtain continuous hitting times, as well as conformal weldings in the classical sense of the term.  We obtain the following.  Note that we state some of these informally; in each case see the referenced result for the precise statement, as well as Remark \ref{Remark:DomainTechnicality} below.
\begin{enumerate}[$(i)$]
    \item\label{Results:DriverToTimes} $\xi \mapsto \tau$ is continuous (Theorem \ref{Cor:TimesPointwise}) but not uniformly so (Lemma \ref{Lemma:DriverToTimesNotUniform}): if $\xi_n$ and $\xi$ generate simple curves and $\xi_n \xrightarrow{u} \xi$, then $\tau_n \xrightarrow{u} \tau$.  However, we can find $\xi_n, \tilde{\xi}_n$ with $\|\xi_n-\tilde{\xi}_n\|_\infty \rightarrow 0$ but where $\|\tau_n - \tilde{\tau}_n  \|_\infty > \epsilon$.  In addition, $(x,\xi) \mapsto \tau(x;\xi)$ is pointwise jointly continuous (Lemma \ref{Lemma:DriverToTimesJoint}).
    \item\label{Results:Lipschitz} $\xi \mapsto \tau^{-1}$ is Lipschitz continuous, with optimal Lipschitz constant 1 (Theorem \ref{Lemma:InverseTauLip}).
    \item $\tau \mapsto \varphi$ is continuous (Lemma \ref{Lemma:TimesToWeldingConvergence}). 
    \item\label{Results:DriverToWelding} $\xi \mapsto \varphi$ is continuous (Theorem \ref{Thm:DriverToWeldingConvergence}), but not uniformly so (Lemma \ref{Lemma:DriverToWeldingNotUniform}). In addition, $(x; \xi) \mapsto \varphi(x;\xi)$ is pointwise jointly continuous (Lemma \ref{Lemma:DriverToWeldingJoint}).
    \item Neither $\varphi \mapsto \tau$ nor $\varphi \mapsto \xi$ is continuous (Theorem \ref{Thm:WeldingToOthersNotContinuous}).
    \item\label{Results:TimesToDriver} $\tau \mapsto \xi$ is well defined and continuous at $\tau_{\tb{C}}$, the hitting time function of the constant driver $\tb{C}(t) \equiv C \in \mathbb{R}$ (Theorem \ref{Thm:TimesToDriverContinuousZero}).  That is, if a driver $\xi$ generates hitting times that are the same as $\tau_{\tb{C}}$, then $\xi = \tb{C}$.  Furthermore, if $\xi_n$ are drivers corresponding to simple curves with hitting times $\tau_n$ satisfying $\tau_n \xrightarrow{u} \tau_{\tb{C}}$, then $\xi_n \xrightarrow{u} \tb{C}$.
\end{enumerate}

\subsubsection{Discussion}
One of the contributions of these results is that no regularity is assumed on the drivers $\xi_n,\xi$ other than they belong to $S([0,T])$, the class of drivers generating simple curves on $[0,T]$.\footnote{There is currently no known analytic characterization of drivers $\xi \in S$, and this remains an important open problem.  The literature on this question, in addition to the above-mentioned work of Lind \cite{Lind4} building off \cite{Marshallrohde}, appears to consist of just \cite{LMR},\cite{Schleissingerthesis} and \cite{Zins}.}   It is well known, however, that the inverse Loewner transform $\xi \mapsto \gamma$ acting on $S$ is not continuous: there exist $\xi_n, \xi \in S([0,T])$ such that $\|\xi_n - \xi\|_{\infty[0,T]} \rightarrow 0$ but where the corresponding curves $\gamma_n$ do not even have subsequential limits in their half-plane capacity parametrizations, let alone uniformly converge \cite[Ex. 4.49]{Lawler}.  Our results in $(\ref{Results:DriverToTimes})$ and $(\ref{Results:DriverToWelding})$ say the uniform topologies on $\tau$ and $\varphi$ are oblivious to this pathological behavior in the $\gamma_n$: the $\tau_n$ and $\varphi_n$ generated by $\xi_n$ still converge to the $\tau$ and $\varphi$ generated by $\xi$.  Furthermore, by $(\ref{Results:Lipschitz})$ one even has quantitative convergence of the points $x_n < 0 <y_n$ welded at a given time $t$ by $\xi_n$ to the points $x< 0 <y$ welded at the same time by $\xi$.

Combined with the known results on $\gamma \mapsto \xi$ (see footnote \ref{Footnote:LoewnerTrans} above), our results show that moving from left to right in \eqref{Intro:Maps} is generally moving from stronger to weaker forms of convergence.  More work is needed to understand the precise nature of the middle arrow.  While our results on its inverse in part  $(\ref{Results:TimesToDriver})$ above are very preliminary, as they only cover constant drivers, we believe the questions behind them are natural: given $\xi \in S([0,T])$, does $\tau(\cdot; \xi)$ determine $\xi$?  And if so, is $\xi \mapsto \tau$ a homeomorphism onto its image?  We find it interesting that, even in the simplest case that we consider, proofs of the existence and continuity of the inverse are not entirely trivial; see \S\ref{Sec:TimesToDriver}.  We attempted to build a driver with large oscillations that we suspected could be a counterexample to continuity of $\tau \mapsto \xi$ for more general $\tau$, but numerical simulations showed our construction still converged for relatively smooth data.  We describe this construction and the simulations in \S\ref{Sec:PositiveEvidence} as positive evidence for a broader result.

We also highlight two other contributions.  In \S\ref{Sec:Zipper} we show that our results for the $\xi \mapsto \varphi$ map in  $(\ref{Results:DriverToWelding})$ imply that, for any finite collections of pairs $\{(x_j,y_j)\}_{j=1}^N$ with
\begin{align*}
    x_N < x_{N-1} < \cdots < x_1 < y_1 < \cdots < y_N,
\end{align*}
there exists a curve of minimal Loewner energy which welds each $x_j$ to $y_j$.  We use this in Theorem \ref{Cor:Zipper} to show that we can well-approximate any given finite-energy curve $\gamma([0,T])$ by an energy minimizer on a sufficiently-fine discretization of its welding.  This is related, although not identical, to the welding zipper algorithm of Donald Marshall.  

Finally, in the appendix we list two integral formulas relating all  our main actors $\xi, \tau$ and $\varphi$ that appear to have thus far escaped notice in the literature.

\begin{remark}\label{Remark:DomainTechnicality}
    Some of the statements in the list above, as alluded to, are imprecise as there is technicality regarding domains to deal with.  For instance, for the $\xi \mapsto \tau$ result in $(\ref{Results:DriverToTimes})$, we assume uniform convergence $\|\xi_n - \xi\|_{\infty[0,T]} \rightarrow 0$ on a fixed time interval $[0,T]$.  However, that does not imply that all the hitting time functions $\tau_n, \tau$ share a common domain (see Example \ref{Eg:NoEndpoints}), and so we prove that if the domains of $\tau_n$ and $\tau$ are $[a_n,b_n]$ and $[a,b]$, respectively, then $a_n \rightarrow a$, $b_n \rightarrow b$ and $\tau_n \xrightarrow{u} \tau$ on any compact subinterval $[c,d] \subset (a,b)$.  (In addition, our result in ($\ref{Results:Lipschitz}$) gives sharp quantitative control on $|a_n-a|$ and $|b_n-b|$.)  
    
    In \emph{this sense} we mean $\xi \mapsto \tau$ is continuous, and many of the other results above are similar.  We are thus often using ``continuity'' in a somewhat informal sense; in particular we do not attempt to equip the space $\tilde{C}$ of continuous functions with different domains with a topology $\tilde{\mathcal{C}}$ which would make, for instance, $\xi \mapsto \tau$ continuous from $(C([0,T]), \|\cdot\|_{\infty[0,T]})$ to $(\tilde{C},\tilde{\mathcal{C}})$, although this may be possible. 
\end{remark}

\subsection{Methods}

For results concerning the $\xi \mapsto \tau$, $\tau \mapsto \varphi$ and $\xi \mapsto \varphi$ maps, we primarily rely on Lemma \ref{Lemma:tauContinuous} combined with the surprising power of Lemma \ref{Lemma:DriverTimes1} and the formula \eqref{Eq:LoewnerShift}. Proofs that maps are not continuous or not uniformly continuous are based on explicitly-constructed examples.  Some new machinery was needed to say anything about the $\tau \mapsto \xi$ direction, and our main tool here is Lemma \ref{Lemma:FasterTimes}, which says that the farther a driver welds two points $x_0 < y_0$ from their initial average $(x_0+y_0)/2$, the lower the hitting time.  

\subsection{Organization}
In \S\ref{Sec:Prelim} we establish notation and review background of deterministic Loewner chains.  We prove our main lemmas in \S\ref{Sec:Lemmas}, the results on $\xi \mapsto \tau$ and $\xi \mapsto \tau^{-1}$ in \S\ref{Sec:DriverToTimes} and \S\ref{Sec:DriverToTimesInverse}, respectively, and give results and simulations regarding $\tau \mapsto \xi$ in \S\ref{Sec:TimesToDriver}.  In \S\ref{Sec:BlankToWelding} we cover continuity of $\tau \mapsto \varphi$ and $\xi \mapsto \varphi$, and show by example in \S\ref{Sec:WeldingCounterexample} that $\varphi \mapsto \tau$ and $\varphi \mapsto \xi$ are not continuous.  Our application of Theorem \ref{Thm:DriverToWeldingConvergence} to minimal-energy curves falls in \S\ref{Sec:Zipper}, and we conclude with open problems in \S\ref{Sec:Problems}, and then the appendix.


\bigskip 
\noindent \textbf{Acknowledgements}
The authors are thankful to Yizheng Yuan for pointing our attention to \cite[Lemma 5.1]{Yizheng} and suggesting how it could yield welding convergence, and for looking at a draft of the paper.  We also thank Don Marshall and Steffen Rohde for looking at a very early draft, and we are grateful to have learned the trick \eqref{Eq:LoewnerShift} from Steffen Rohde (perhaps it goes back to Oded Schramm), and to have seen it applied in a similar manner, albeit rougher, to what we do in Theorem \ref{Lemma:InverseTauLip}. This research was partially conducted while the authors were at Mathematical Sciences Research Institute during the spring 2022 semester and is thus partially supported by the US National Science Foundation under Grant No. DMS-1928930.

\section{Notation and preliminaries}\label{Sec:Prelim}
\subsection{Basics of Loewner theory}
We sketch some notation and results concerning the Loewner equation; for more background see, for instance,  \cite{Kemp} or \cite{Lawler}.  We frame the theory largely in terms of the reverse/upwards maps $h_t$, as they induce the hitting times $\tau$.

Indeed, $h_t:\mathbb{H} \rightarrow \mathbb{H} \backslash \gamma_t([0,t])$ in \eqref{Eq:UpwardsLoewner} is the unique conformal map which fixes $\infty$ and satisfies
\begin{align}\label{Eq:UpwardsMapsInfty1}
    h_t(z) = z + O(1/z), \qquad z \rightarrow \infty.
\end{align}
In other words, scaling and translation only occurs locally around $\gamma_t$, not at $\infty$.  We assume $\gamma_t$ is parametrized by \emph{half-plane capacity}, in which case the above expansion is actually
\begin{align}\label{Eq:h_tInfty}
    h_t(z) = z - \frac{2t}{z} + O(1/z^2), \qquad z \rightarrow \infty,
\end{align}
and we say that the half-plane capacity $\hcap(\gamma_t)$ of $\gamma_t$ is $2t$ (so note time $t$ corresponds to $\hcap$ $2t$).  The extension of $h_t(z)$ to the real line maps $\xi(0)$ to the tip $\tilde{\gamma}_t(t)$ of the curve $\tilde{\gamma}_t$ generated by $\xi$ on $[0,t]$, while the base of $\gamma_t$ is at $\xi(t)$.

Note that can define $\hcap(K)$ whenever $K\subset \mathbb{H}$ is a \emph{compact $\mathbb{H}$-hull}, which is to say, $K$ is bounded, relatively closed in $\mathbb{H}$, and $\mathbb{H} \bs K$ is simply connected.  In this case, there is again a unique conformal map $h_K:\mathbb{H} \rightarrow \mathbb{H}\bs K$ satisfying \eqref{Eq:UpwardsMapsInfty1} at $\infty$, and with expansion
\begin{align*}
    h_K(z) = z - \frac{a_1}{z} + O(1/z^2), \qquad z \rightarrow \infty.
\end{align*}
We define $\hcap(K) := a_1$, which is positive when $K \neq \emptyset$.  It follows from the definition that $\hcap$ satisfies $\hcap(K+x) = \hcap(K)$ and $\hcap(rK) = r^2\hcap(K)$ for $r \geq 0$.  Also, $K_1 \subset K_2$ implies $\hcap(K_1) \leq \hcap(K_2)$.  Considering $K= \overline{B_1(0)} \cap \mathbb{H}$ as a concrete example, we find that $h_K^{-1}(z) = z+1/z$, and thus
\begin{align}\label{Eq:Diskhcap}
    \hcap\big(\overline{B_1(0)} \cap \mathbb{H} \big) = 1.
\end{align}
There is also a stochastic definition, which allows us to drop the requirement that $\mathbb{H}\bs K$ be simply connected.  Indeed, one can show
\begin{align}\label{Def:hCapProbability}
    \hcap(K) = \lim_{y \rightarrow \infty}y \mathbb{E}^{iy}(\text{Im } B_\tau),
\end{align}
where $B_t$ is two-dimension Brownian motion started from $iy$, and $\tau$ is the exit time of $\mathbb{H} \bs K$.  See \cite[\S3.4]{Lawler} for these and further properties.

We will interchangeably call the dynamics given by \eqref{Eq:UpwardsLoewner} the ``upwards Loewner flow'' and ``reverse Loewner flow.''  The former is not entirely standard, but is natural from the point of view that the curves $\gamma_t$ grow upwards into $\mathbb{H}$ from $\xi$'s position in $\mathbb{R}$.  The \emph{downwards} or \emph{forwards} map $g_t$ is the unique map from $\mathbb{H}\bs \gamma([0,t])$ to $\mathbb{H}$ which fixes $\infty$ and is $z + O(1/z)$ near infinity, and in this direction the expansion corresponding to \eqref{Eq:h_tInfty} is
\begin{align*}
    g_t(z)= z + \frac{2t}{z} + O(1/z^2), \qquad z \rightarrow \infty,
\end{align*}
and the Loewner equation for the $g_t$ is 
\begin{align}\label{Eq:Loewner}
    \dot{g}_t(z) = \frac{2}{g_t(z)-\lambda(t)}, \qquad g_0(z) = z.
\end{align}

The relation between the $h_t$ and $g_t$, $\xi$ and $\lambda$, and $\gamma_t$ and $\gamma$ is the following.  Let $\gamma:[0,T] \rightarrow \mathbb{H} \cup \{x\}$ be a fixed, simple curve with $\gamma(0)=x$ (which is what our notation for the range means).  The downwards/forward driving function $\lambda$ is just the reversal of $\xi$, $\lambda(t) = \xi(T-t)$.  (If we wish to normalize by starting at zero, we may take $\lambda(t) = \xi(T-t)-\xi(T)$, or equivalently, $\xi(t) = \lambda(T-t) - \lambda(T)$.)  We always write $\lambda$ from the downwards driving function and $\xi$ for the upwards.  The curve $\gamma_t$ generated by $\xi$ on $[0,t]$ is the conformal image $g_{T-t}\big( \gamma([T-t,T]) \big)$ of the last $t$ units $\gamma([T-t,T])$ of $\gamma$ under $g_{T-t}$.  That is,
\begin{align*}
    h_t = g_{T-t} \circ g_T^{-1}, \qquad 0 \leq t \leq T.
\end{align*}
To see this, note that it holds for $t=0$, and observe $\partial_t (g_{T-t}^{-1} \circ h_t ) =0$ by  \eqref{Eq:UpwardsLoewner} and \eqref{Eq:Loewner}.  So for the $g_t$ the underlying $\gamma$ is a fixed curve which is growing at its tip, whereas in the case of the $h_t$ maps, $\gamma_t$ grows at its base and the entire curve is constantly being conformally deformed.

Recall driving functions $\xi, \lambda$ are always continuous and thus members of $C([0,T])$.  We write $C_0$ for the set of continuous functions starting at zero, and $S, S_0$ for those $\xi \in C,C_0$ that generate a \emph{simple curve} $\gamma^\xi$ upwards Loewner flow (where we include the time domain $[0,T]$ as needed).  By this we mean that the final curve $\gamma^\xi$ generated on time $[0,T]$ by $\xi$ is non self-intersecting and also does not touch $\mathbb{R}$ other than at its base: $\gamma^\xi \cap \mathbb{R} = \{\xi(T)\}$.  It is not hard to see this is equivalent to saying that the curve generated by $\xi$ on any time interval $[t_1,t_2] \subset [0,T]$ has these same two properties.

Recall that well-known elements of $S$ include linear drivers \cite{KNK}, drivers with one-sided H\"{o}lder-1/2 norm less than four \cite{Lind4,LMR,Zins} and, a.s., scaled Brownian motion $\sqrt{\kappa}B(t)$ when $0 \leq \kappa \leq 4$ \cite{BasicpropSLE}.

Points $x_0< \xi(0) <y_0$ under \eqref{Eq:UpwardsLoewner} flow along the real line towards $\xi$, and we interchangeably write
\begin{align*}
    x(t) = x(t;\xi) = h_t(x_0;\xi) = h_t(x_0)
\end{align*}
for the image of $x_0$ after $t$ units of time under driver $\xi$, and similarly for $y(t)$.

\subsection{Hitting times and weldings}

Let $\xi$ be continuous.  For $x_0 \in \mathbb{R} \bs \xi(0)$, the dynamics in \eqref{Eq:UpwardsLoewner} becomes 
\begin{align}\label{Eq:UpwardsLoewnerx}
    \dot{x}(t) = \frac{-2}{x(t) - \xi(t)}, \qquad x(0) = x_0. 
\end{align}
As noted in the introduction, the \emph{hitting time} $\tau(x_0)$ of $x_0$ is then
\begin{align}\label{Def:HittingTime1+}
    \tau(x) = \inf\{\,t \geq 0 \; : \; \liminf_{s \nearrow t} |x(s) - \xi(s)| =0  \,\},
\end{align}
and so the flow $t \mapsto x(t;\xi)$ is well defined on $[0,\tau(x_0))$.  Note that if there are no such times $t$ for some $x$, then $\tau(x) = +\infty$, and we say that $x$ is not welded by $\xi$.  We also remark that   \eqref{Def:HittingTime1+} says nothing about whether or not $\xi$ generates a simple curve; $\tau(x)$ is well defined in either case.

Suppose $\tau(x_0) <\infty$.  By \eqref{Eq:UpwardsLoewnerx}, $x_0$ flows monotonically towards $\xi$, and its position is bounded by the maximum of $\xi$ on $[0,\tau(x_0)]$.  In particular, $\lim_{t\nearrow \tau(x_0) }x(t)$ exists, and by \eqref{Def:HittingTime1+} must be $\xi(\tau(x_0))$.  Thus we can extend $t \mapsto x(t)$ from $[0,\tau(x_0))$ to $[0,\tau(x_0)]$ by setting $x(\tau(x_0)) = \xi(\tau(x_0))$, and we have
\begin{align}
    \tau(x) &= \inf\{\,t \geq 0 \; : \; \lim_{s \nearrow t} |x(s) - \xi(s)| =0  \,\} \notag\\
    &= \inf\{\,t \geq 0 \; : \;  x(t)=\xi(t)  \,\}.\label{Def:HittingTime3}
\end{align}


In the case of the zero driver $\tb{0}(t) \equiv 0$, for example, it is not hard to see that the map satisfying \eqref{Eq:UpwardsLoewner} is $h_t(z) = \sqrt{z^2 - 4t}$, and so the points mapping to the base of the curve 0 under the extension of $h_t$ to $\mathbb{R}$ are $\pm 2\sqrt{t}$, yielding the hitting times
\begin{align}\label{Eq:ZeroHittingTime}
    \tau(x; \tb{0})=\frac{x^2}{4}, \qquad x \in \mathbb{R}.
\end{align}

We show in Lemma \ref{Lemma:tauContinuous} that, for $\xi \in S$, $x \mapsto \tau(x)$ is strictly increasing as one moves away from $\xi(0)$.  We can thus think of $\tau$ as consisting of two invertible functions, the left and the right hitting times, which we denote by
\begin{align}\label{Eq:HittingTimesTwoFunctions}
    \tau_- := \tau|_{x\leq \xi(0)} \qquad \text{ and } \qquad \tau_+ := \tau|_{y \geq \xi(0)}.  
\end{align}

As mentioned in the introduction, the \emph{conformal welding} associated to $\xi \in S$ is the homeomorphism of intervals on either side of $\xi(0)$ which satisfies $\tau(x) = \tau(\varphi(x))$ for all $x$ with $\tau(x)<\infty$.  We take the convention that $\varphi$ maps from the left of $\xi(0)$ to the right, and so more precisely, for $x \leq \xi(0)$,
\begin{align*}
    \tau_-(x) = \tau_+(\varphi(x)) \qquad \text{ or } \qquad \tau_+^{-1} \circ \tau_-(x) = \varphi(x).
\end{align*}
The welding can also be defined in terms of the maps $h_t$ via $h_t(\varphi(x)) = h_t(x)$ for $t \geq \tau(x)$.

\subsection{Other notation}
For a curve welding $\varphi:[-a,0] \rightarrow [0,b]$, when one (or both) of $-a,b$ is infinite, we mean that $\varphi$ is defined on $(-a,0]$ and $\lim_{x \rightarrow a^+}\varphi(x) = b$.  

For a continuous function $f$ on an interval $[a,b]$, we write $\|f\|_{\infty[a,b]} := \max_{x \in [a,b]}|f(x)|$.  Of course, $f_n \xrightarrow{u} f$ on $[a,b]$ means $\|f_n-f\|_{\infty[a,b]} \rightarrow 0$.   

We write $a \wedge b := \min\{a,b\}$ and $a \vee b := \max\{a,b\}$, and $A(x) \asymp B(x)$ means there exists constant $C\geq 1$ such that
\begin{align*}
    \frac{1}{C}B(x) \leq A(x) \leq CB(x) 
\end{align*}
for all values of $x$.

\subsection{Elementary lemmas}
We will use the following two easy lemmas.  The first is similar to Dini's theorem and specifies a situation where we can upgrade from pointwise to uniform convergence. 
\begin{lemma}\label{Lemma:OurDini}
    Let $[a,b] \subset \mathbb{R}$ be a closed interval and $f_n:[a,b] \rightarrow \mathbb{R}$ a sequence of functions such that each $f_n$ is non-increasing or non-decreasing.  Suppose $f:[a,b] \rightarrow \mathbb{R}$ is continuous and $f_n \rightarrow f$ pointwise on $[a,b]$.  Then we have $\|f_n-f \|_{\infty[a,b]} \rightarrow 0$.
\end{lemma}
\noindent Note that we do not assume the $f_n$ are continuous or that the monotonicity across the sequence is the same, i.e. some $f_n$ may be increasing and some decreasing.
\begin{proof}
    Choose $\delta>0$ such that $|f(x)-f(y)|<\epsilon/2$ whenever $a \leq x,y \leq b$ satisfy $|x-y|<\delta$, and let $a = x_0 < x_1< \cdots <x_m = b$ be a partition of $[a,b]$ of mesh size less than $\delta$. Let $N$ be large enough so that $|f_n(x_j)-f(x_j)|<\epsilon/2$ for all $j \in \{1,\ldots, m\}$ whenever $n \geq N$.  Choosing such an $n$, suppose $f_n$ is non-decreasing.  Then for any $x_j < x < x_{j+1}$,
    \begin{align*}
        f_n(x) - f(x) \leq f_n(x_{j+1}) - f(x_{j+1}) +f(x_{j+1}) - f(x) < \epsilon,
    \end{align*}
    and similarly
    \begin{align*}
        f(x)-f_n(x) \leq f(x) - f(x_j) + f(x_j) - f_n(x_j) < \epsilon.
    \end{align*}
    The argument is similar when $f_n$ is non-increasing.
\end{proof}

\begin{lemma}\label{Lemma:InversesConverge}
    Let $f_n, f\colon \mathbb{R} \to \mathbb{R}$ be continuous and strictly increasing with $f_n \to f$ point-wise. Then $f_n^{-1} \to f^{-1}$ point-wise.
\end{lemma}

\begin{proof}
    Let $y \in \mathbb{R}$ and $x = f^{-1}(y)$. By the monotonicity of $f$ we have $f(x-\varepsilon) < f(x) < f(x+\varepsilon)$. Let $\delta := |f(x)-f(x-\varepsilon)| \vee |f(x)-f(x+\varepsilon)|$ and let $n$ be large enough such that $|f(x\pm\varepsilon)-f_n(x\pm\varepsilon)| < \delta$ . Then $f_n(x-\varepsilon) < f(x) < f_n(x+\varepsilon)$, and consequently $f_n^{-1}(y) \in {(x-\varepsilon,x+\varepsilon)}$ by monotonicity of $f_n$.
\end{proof}

\section{Continuity properties of $\xi \mapsto \tau$ and $\xi \mapsto \tau^{-1}$}\label{Sec:TauContinuity}

\subsection{Lemmas}\label{Sec:Lemmas}
We establish some tools before proving our continuity results.  Our first lemma gives basic properties of the hitting times $x \mapsto \tau(x;\xi)$ when $\xi \in S$.  This is a generalization of \cite[Prop. 4.1$(a)$]{Vlad} and \cite[Lemma 3]{Lind4}; the former, because this always holds, rather than only almost surely, and the latter, because we only assume $\xi \in S([0,T])$, not that $\xi \in \text{H\"{o}l}(1/2)$ with $|\xi|_{1/2}<4$.
\begin{lemma}\label{Lemma:tauContinuous}
    If $\xi \in S([0,T])$, $0 < T \leq \infty$,  $x \mapsto \tau(x) = \tau(x;\xi)$ is continuous.  It is strictly increasing for $x\geq\xi(0)$ and strictly decreasing for $x \leq \xi(0)$. In particular, for each $0 <t \leq T$ there are exactly two points $x < \xi(0) < y$ such that $\tau(x;\xi) = t = \tau(y;\xi)$.
\end{lemma}

\begin{proof}
    Without loss of generality $\xi(0)=0$, and by symmetry it suffices to prove continuity and monotonicity for $y> 0$.  If $0 <y_1 < y_2$ and $t < \tau(y_1)$, then \eqref{Eq:UpwardsLoewner} yields
    \begin{align}\label{Eq:LoewnerIncrement}
        \frac{d}{dt} \big(y_2(t)-y_1(t) \big) = \frac{2(y_2(t)-y_1(t))}{(y_2(t) - \xi(t))(y_1(t)-\xi(t))} > 0.
    \end{align}
    Thus the points are getting further apart for $t<\tau(y_1)$, showing $y_1(\tau(y_1)) = \xi(\tau(y_1)) < y_2(\tau(y_1))$.  Since $\xi$ is continuous, $\tau(y_1) < \tau(y_2)$, and we see $y \mapsto \tau(y)$ is strictly increasing.
    
In particular, $\tau$ can only have jump discontinuities.  To see this actually does not happen, pick $0 <y_1 < y_3$ arbitrarily and let $t_2 \in (\tau(y_1), \tau(y_3))$.  We show there exists $y_2$ with $\tau(y_2) = t_2$.  Indeed, map up with $h_{t_2}$.  Since the curve $\gamma_{t_2}$ generated by $\xi([0,t_2])$ is simple, the prime end $\xi(t_2)$ of $\mathbb{H} \backslash \gamma_{t_2}$ corresponds to exactly two pre-images $x_2 < 0 < y_2$ under the extension of $h_{t_2}$.  Since 
\begin{align}\label{Eq:x2hit}
    h_{t_2}(y_2) = \xi(t_2),
\end{align}
$\tau(y_2) \leq t_2$ by definition of the hitting time.  

Suppose that $y_2$'s hitting time is actually earlier, i.e.\begin{align}\label{Eq:x2Assumption}
    \tau(y_2) =: t_2' < t_2,
\end{align}
which means we also have that $h_{t_2'}(y_2) = \xi(t_2')$.  Let $\tilde{h}_{t}$ be the map which, for $t \geq t_2'$, solves
\begin{align*}
    \partial_t \tilde{h}_t(z) = \frac{-2}{\tilde{h}_t(z) - \xi(t)}, \qquad \tilde{h}_{t_2'}(z) = z,
\end{align*}
i.e. which maps to the complement of the curve segment generated by $\xi$ for times $t \geq t_2'$.  Then for $t>t_2'$, $\tilde{h}_t$ sends $\xi(t_2')$ to the tip of the non-trivial curve segment generated on $[t_2',t]$ by $\xi$, which thus has positive imaginary part by the assumption $\xi \in S$ and the fact that its capacity is $2(t-t_2')$ (see \cite[Lemma 4.2]{Kemp} for the latter).  Since $h_{t_2} = \tilde{h}_{t_2-t_2'} \circ h_{t_2'}$, we thus see $\text{Im}\,h_{t_2}(y_2)>0$, which contradicts \eqref{Eq:x2hit}.  We conclude that $\tau(y_2)=t_2$ and that there are no jump discontinuities in $\tau$.

 

The last statement follows from the strict monotonicity and definition of the hitting times \eqref{Def:HittingTime3}.
\end{proof}

We will need the following two observations stemming from the Loewner equation \eqref{Eq:UpwardsLoewner}.  First, for fixed $\delta \in \mathbb{R}$, we have the shifting formula
    \begin{align}\label{Eq:LoewnerShift}
        h_t(x+\delta; \xi+\delta) = h_t(x; \xi)+\delta.
    \end{align}
    This holds for any $x \in \mathbb{H}$ for all $t$ (since solutions starting in $\mathbb{H}$ last forever) and for $x \in \mathbb{R}$ when $0 \leq t \leq \tau(x)$.
    
    Secondly, given two drivers $\xi_1 \leq \xi_2$ and a point $y_0$ with $\xi_2(0) < y_0$ with a time $t \leq \tau(y_0; \xi_2)$, driver monotonicity combined with the Loewner equation yield
    \begin{align}\label{Ineq:LoewnerOrdering}
        h_t(y_0; \xi_2) \leq h_t(y_0; \xi_1).
    \end{align}

The following lemma gives the key inequality for our $\xi \mapsto \tau$ arguments.  While we borrow the statement from \cite{Yizheng}, we offer a new, succinct proof, and also drop the requirement that the drivers start at the same location.  
\begin{lemma}\label{Lemma:DriverTimes1}
    \cite[Lemma 5.1]{Yizheng} Suppose $\xi, \tilde{\xi} \in C([0,T])$ with $\|\xi-\tilde\xi\|_{\infty[0,T]} \le \delta$. Fix $y$ with $\tilde{\xi}(0) < y$.  If $t < \tau(y;\tilde\xi)$ then $t< \tau(y+\delta; \xi)$.  Similarly, for $x < \tilde{\xi}(0)$, if $t< \tau(x; \tilde{\xi})$ then $t < \tau(x-\delta; \xi)$.
\end{lemma}
\noindent Note also that, as in \cite{Yizheng}, we do not assume the drivers generate simple curves.
\begin{proof}
    
    
    By symmetry, it suffices to prove the statement for $y$ with $\tilde{\xi}(0) < y$.  Observing $t < \tau(y;\tilde{\xi})$ is equivalent to $\tilde{\xi}(t) < h_t(y; \tilde{\xi})$, we use \eqref{Eq:LoewnerShift} and \eqref{Ineq:LoewnerOrdering} to conclude
    \begin{equation*}
        \xi(t) \leq \tilde{\xi}(t) + \delta < h_t(y; \tilde{\xi}) + \delta = h_t(y+\delta; \tilde{\xi}+\delta) \leq h_t(y+\delta; \xi),
    \end{equation*}
    and hence $t < \tau(y+\delta; \xi)$.
\end{proof}
\begin{corollary}\label{Cor:TimesSandwich}
    Suppose $\xi, \tilde{\xi} \in C([0,T])$, $0 < T \leq \infty$, with $\|\xi-\tilde\xi\|_{\infty[0,T]} \le \delta$.  Then for any $y > \xi(0)+\delta$,
    \begin{align}\label{Ineq:DriverTimesMonotone}
        \tau(y-\delta ; \xi) \leq \tau(y; \tilde\xi) \leq \tau(y+\delta ; \xi).
    \end{align}
    Similarly, for $x < \xi(0)-\delta$,  $\tau(x-\delta ; \xi) \geq \tau(x; \tilde\xi) \geq \tau(x+\delta ; \xi).$
\end{corollary}
\noindent As no assumption is made on the finiteness of the hitting times, part of the statement of the corollary is that \eqref{Ineq:DriverTimesMonotone} holds whether or not each of the $\tau$'s is finite or infinite.  Note also that we again do not assume that $\xi, \tilde{\xi} \in S$.
\begin{proof}
By symmetry it suffices to consider $y > \xi(0) + \delta$.  If $\tau(y-\delta; \xi) = \infty$, then by using $t = n$ for large $n$ in Lemma \ref{Lemma:DriverTimes1} we find $n<\tau(y;\tilde{\xi})$, and so $\tau(y; \tilde{\xi}) = \infty$ as well.  Thus \eqref{Ineq:DriverTimesMonotone} holds when at least two of the members are infinite.

Suppose that $\tau(y; \tilde{\xi}) < \infty$.  Since $y>\tilde{\xi}(0)$ by assumption, by using times $\tilde{t}_n = \tau(y;\tilde{\xi})-1/n$ in Lemma \ref{Lemma:DriverTimes1} we obtain the right-most inequality in \eqref{Ineq:DriverTimesMonotone} in the limit.  The assumption $\tau(y; \tilde{\xi}) < \infty$ and Lemma \ref{Lemma:DriverTimes1} also yield $\tau(y-\delta; \xi) <\infty$, and then exchanging the roles of $\tilde{\xi}$ and $\xi$ and replacing $y$ with $y-\delta$ in the above argument yields the left inequality as well.  

We conclude \eqref{Ineq:DriverTimesMonotone} holds irrespective of whether any of the three hitting times is finite or infinite.
\end{proof}


\subsection{Continuity of $\xi \mapsto \tau$}\label{Sec:DriverToTimes}

Our above inequalities immediately yield continuity properties of the driver-to-hitting-time map.  Here we restrict to $\xi \in S([0,T])$ to avail ourselves of the continuity of $\tau$ from Lemma \ref{Lemma:tauContinuous}.

\begin{theorem}\label{Cor:TimesPointwise}
    Let $0<T \leq \infty$ and $\xi, \xi_n \in S([0,T])$ be drivers, with $[a,b]$ the maximal interval on which $\tau(\cdot; \xi)$ is finite, and $[a_n,b_n]$ the maximal interval where $\tau(\cdot; \xi_n)$ is finite, $-\infty \leq a < b \leq \infty$, $-\infty \leq a_n < b_n \leq \infty$.  If $\xi_n \xrightarrow{u} \xi$ on $[0,T]$, then $a_n \rightarrow a$, $b_n \rightarrow b$, and $\tau(x;\xi_n) \xrightarrow{u} \tau(x;\xi)$ on any $[c,d] \subset (-a,b)$.  
\end{theorem}
\noindent With regards to the technicality that we cannot include the endpoints $-a$ and $b$ in the convergence, see Example \ref{Eg:NoEndpoints} below.  We give sharp quantitative control on how far $a_n, b_n$ can be from $a,b$ in Theorem \ref{Lemma:InverseTauLip} below.

The notation is slightly imprecise when one of the interval endpoints is $\infty$. For instance, if $a =-\infty$ and $b \in \mathbb{R}$, we mean that $\tau(\cdot; \xi)$ is finite on $(-\infty,b)$ and $a_n$ is either identically $-\infty$ or is eventually less than any $-N$ for all large $n$.  If $T=\infty$, we mean $\xi,\xi_n$ are defined on $[0,\infty)$.
\begin{proof}
    Fix an interval $[c,d] \subset (a,b)$ and $\epsilon >0$, and choose $\delta$ small enough such that the following three conditions hold: $[c-\delta, d+\delta] \subset (a,b)$, 
    \begin{align}\label{Ineq:TimesNearZero}
        \max\{\, \tau(\xi(0)-2\delta;\xi), \tau(\xi(0)+2\delta;\xi) \,\} < \epsilon,
    \end{align}
    and the modulus of continuity $\omega$ of $\tau(\cdot; \xi)$ on $[c-\delta,d+\delta]$ satisfies
    \begin{align}
        \omega(\delta; [c,d]) &:= \sup\{\, |\tau(u;\xi) - \tau(v;\xi)| \; : \; c-\delta\leq u,v \leq d+\delta, \, |u-v| \leq \delta \,\}\notag\\ 
        &< \epsilon.\label{Ineq:TimesModulusSmall}
    \end{align}
    By Corollary \ref{Cor:TimesSandwich}, the first condition yields that $\tau(x;\xi_n) < \infty$ for all $x \in [c,d]$ when $n$ is large enough.  For $x \in [c,\xi(0)-\delta) \cup (\xi(0)+\delta, d]$, by \eqref{Ineq:DriverTimesMonotone} we then have
    \begin{align*}
        |\tau(x;\xi_n) - \tau(x;\xi)| \leq \omega(\delta; [c,d]) < \epsilon.
    \end{align*}
    If $\xi(0)-\delta \leq x \leq \xi(0) + \delta$, then
    \begin{align*}
        |\tau(x;\xi_n) - \tau(x;\xi)| \leq \tau(x;\xi_n) + \tau(x;\xi) \leq \tau(x;\xi_n) + \epsilon
    \end{align*}
    by \eqref{Ineq:TimesNearZero}.  Furthermore, using monotonicity and the fact that $\xi_n(0) \in [\xi(0)-\delta, \xi(0)+\delta]$ for all large $n$, we see
    \begin{align*}
        \tau(x;\xi_n) &\leq \max\{\, \tau(\xi(0)-\delta; \xi_n), \tau(\xi(0)+\delta; \xi_n) \,\}\\ &\leq \max\{\, \tau(\xi(0)-2\delta;\xi), \tau(\xi(0)+2\delta;\xi) \,\} < \epsilon
    \end{align*}
    by \eqref{Ineq:DriverTimesMonotone} and \eqref{Ineq:TimesNearZero}.  We conclude $\tau(\cdot;\xi_n)$ is uniformly close to $\tau(\cdot;\xi)$ on $[c,d]$ for large $n$.

    
    For the welding interval endpoints, by symmetry it suffices to show that $b_n \rightarrow b$.  Suppose first that $b<\infty$ and let $\epsilon >0$.  Since by the above $\tau(b-\epsilon;\xi_n) \rightarrow \tau(b-\epsilon;\xi) < \infty$, monotonicity of the hitting times yields $b -\epsilon \leq b_n$ for large $n$, and thus $b \leq \liminf_{n \rightarrow \infty} b_n$.  On the other hand, however,
    \begin{align*}
        +\infty = \tau(b+\epsilon/2; \xi) \leq \tau(b + \epsilon/2 + \| \xi_n -\xi\|_\infty ; \xi_n) \leq \tau(b + \epsilon; \xi_n)
    \end{align*}
    for all large $n$, where the first inequality is by \eqref{Ineq:DriverTimesMonotone}.  Thus $b_n \leq b+\epsilon$ for all large $n$, showing $\limsup_{n \rightarrow \infty} b_n \leq b$.
    
    If $b=+\infty$, then by the first paragraph we have $\tau(x; \xi_n) \rightarrow \tau(x;\xi)$ for any fixed, finite $x>\xi(0)$.  This implies $b_n \geq x$ for all large $n$, and thus $b_n \rightarrow \infty$.
\end{proof}

\begin{example}\label{Eg:NoEndpoints}
    We cannot conclude that the convergence of hitting times in Theorem \ref{Cor:TimesPointwise} extends all the way to the endpoints of the welding interval $[a,b]$, since, for instance, we may have $b_n <b$ for all $n$.  Consider, for example, $\xi \equiv 0$ on $0 \leq t \leq 1/4$, which generates the vertical line segment $\gamma = [0,i]$.  For the curves generating $\xi_n$, set $\alpha_n := \frac{1}{2}-\frac{1}{n}$ and consider straight line segments $\gamma_n$ from $0$ to
    \begin{align*}
        \gamma_n(\tau_n) = \alpha_n^{\alpha_n-\frac{1}{2}}(1-\alpha_n)^{\frac{1}{2}-\alpha_n} \,e^{i \alpha_n \pi} = \Big(1 + \frac{4}{n^2} + O(n^{-4}) \Big)e^{i \alpha_n \pi}.
    \end{align*}
    The conformal map $F_n: \mathbb{H} \rightarrow \mathbb{H} \backslash \gamma_n([0,\tau_n])$ taking 0 to the tip $\gamma_n(\tau_n)$ which satisfies $F_n(z) = z + O(1)$ as $z \rightarrow \infty$ is explicitly
    \begin{align*}
        F_n(z) &= \Big(z- \sqrt{\frac{\alpha_n}{1-\alpha_n}} \Big)^{\alpha_n}\Big(z + \sqrt{\frac{1-\alpha_n}{\alpha_n}} \Big)^{1- \alpha_n}\\
        &= z + \frac{4}{\sqrt{n^2-4}} - \frac{1}{2z} + O(z^{-2}), \qquad z \rightarrow \infty
    \end{align*}
    (see \cite[The Slit Algorithm]{MRZip}, for instance). From this see that the centered welding $\varphi_n:[-a_n,0] \rightarrow [0,b_n]$ for $\gamma_n$ has endpoints 
    \begin{align}\label{Eq:EndptsPathology}
        -a_n = -\sqrt{\frac{1-\alpha_n}{\alpha_n}}, \qquad b_n = \sqrt{\frac{\alpha_n}{1-\alpha_n}} = \sqrt{\frac{1/2 -1/n}{1/2+1/n}}<1,
    \end{align}
    that the total time for $\gamma_n$ is $\tau_n \equiv 1/4$ and that the driver $\lambda_n$ for $\gamma_n$ has terminal value $\lambda_n(1/4)= 4/\sqrt{n^2-4}$.  As it is well-known that the driver is $\lambda_n(t) = C_n\sqrt{t}$ (see \cite[Example 4.12]{Lawler}, for instance), $\lambda_n$ monotonically increases in $t$ and so the reversed drivers $\xi_n(t) := \lambda_n(1/4-t) -\lambda_n(1/4) \in C_0([0,1/4])$ converge uniformly to zero on $[0,1/4]$ as $n \rightarrow \infty$.  However, as they only weld $[-a_n,b_n]$, we see 
    \begin{align*}
        +\infty \equiv \tau(1; \xi_n) \not\rightarrow \tau(1; \xi) = 1/4.
    \end{align*}
    \end{example}


While the mapping $\xi \mapsto \tau$ is continuous in the sense of Theorem \ref{Cor:TimesPointwise}, the next lemma says there is no global modulus of continuity.
\begin{lemma}\label{Lemma:DriverToTimesNotUniform} 
    There exist two sequences of drivers $\xi_n, \tilde{\xi}_n \in S_0([0,T])$ such that $\|\xi_n - \tilde{\xi}_n\|_{\infty[0,T]} \rightarrow 0$ but where for some points $y_n$ and fixed $\epsilon>0$, \begin{align*}
        \tau(y_n;\xi_n)<\infty, \quad  \tau(y_n;\tilde{\xi}_n)<\infty, \quad \text{ and } \quad |\tau(y_n;\xi_n)- \tau(y_n;\tilde{\xi}_n)| \geq \epsilon
    \end{align*}
    for all $n$.
\end{lemma}
\begin{proof}
    We first construct $\xi_n \in S_0$ and $\tilde{\xi}_n \in S$, and then modify the construction so that all drivers are in $S_0$.  Indeed, set
    \begin{align*}
        \xi(t) = \xi(t,\delta) := \begin{cases}
            -t/\delta & 0 \leq t \leq \delta\\
            -1 & \delta <t,
        \end{cases}
    \end{align*}
    and $\tilde{\xi}(t) = \tilde{\xi}(t,\delta) := \xi(t) - \delta$, and consider $y_0 = y_0(\delta) := 2\delta$.  Note that $\xi, \tilde{\xi} \in S$ since $\xi$ is piecewise linear.  It is easy to see from the Loewner equation \eqref{Eq:UpwardsLoewner} that on $0\leq t \leq \delta$,
    \begin{align*}
        y_0(t):= y_0(t;\xi) = 2\delta - \frac{t}{\delta},
    \end{align*} 
    and thus $y_0(\delta) = -1 + 2\delta$, yielding
    \begin{align}\label{Eq:NotUniformFastTime}
        \tau(y_0; \xi) = \delta + \frac{(2\delta)^2}{4} = \delta + \delta^2
    \end{align}
    by \eqref{Eq:ZeroHittingTime}.  On the other hand, for $\tilde{y}_0(t) := y_0(t;\tilde{\xi})$, it is not hard to see that, as $y_0 - \tilde{\xi}(0) = 3\delta$, $\tilde{y}_0(t)-\tilde{\xi}(t)$ is increasing on $[0,\delta]$ for sufficiently-small $\delta$, and so we have the coarse estimate
    \begin{align}\label{Ineq:MovementOfy0}
        \Delta \tilde{y}|_{[0,\delta]} \leq \frac{2}{3\delta}\delta = \frac{2}{3},
    \end{align}
    yielding
    \begin{align}\label{Ineq:TimeBoundedBelow}
        \tau(y_0; \tilde{\xi}) \geq \delta + \frac{1}{4}\Big( \frac{1}{3} + 2\delta\Big)^2 \geq \frac{1}{36}.
    \end{align}
    Thus setting $\xi_n(t) := \xi(t,1/n)$, $\tilde{\xi}_n(t) := \tilde{\xi}(t,1/n)$, and $y_n:=y_0(1/n)$ we have $|\tau(y_n;\xi_n)- \tau(y_n;\tilde{\xi}_n)|$ bounded below while the drivers become arbitrarily close.
    
    We can easily adjust this construction so that $\tilde{\xi}(\cdot,\delta) \in S_0$ by starting $\tilde{\xi}$ at zero and having it move in time $\delta^2$, say (i.e. extremely fast), to $\delta$, while $\xi(\cdot, \delta)$ remains at zero for $0 \leq t \leq \delta^2$.  We then proceed as in the above construction, setting $y_0$ to be the point which has image $y(\delta^2;\xi)=2\delta$ under $\xi$.  Then $y(\delta^2;\xi) < y(\delta^2;\tilde{\xi})$ and so we still obtain \eqref{Ineq:TimeBoundedBelow}, while from \eqref{Eq:NotUniformFastTime}, $\tau(y_0;\xi) = \delta + 2\delta^2$.
\end{proof}


We also easily have the following pointwise joint continuity of $(x, \xi) \mapsto \tau(x;\xi)$.  

\begin{lemma}\label{Lemma:DriverToTimesJoint}
    Let $0<T \leq \infty$ and $\xi \in S([0,T])$ be a driver, and suppose $\tau(\cdot; \xi)$ is finite on the interval $[-a,b]$.  If $x \in (a,b)$ and $\epsilon>0$, there exists $\delta = \delta(\epsilon, x,\xi)$ such that whenever $|\tilde{x}-x| < \delta$ and $\tilde{\xi} \in S([0,T])$ satisfies $\|\tilde{\xi} - \xi\|_{\infty [0,T]} < \delta$, 
    \begin{align*}
        |\tau(\tilde{x};\tilde{\xi}) - \tau(x; \xi)| < \epsilon.
    \end{align*}
\end{lemma}
\begin{proof}
    Let $\omega(\cdot; \tau)$ be the modulus of continuity of the uniformly-continuous function $\tau(\cdot; \xi)$ on $[a,b]$, and suppose first that $x=\xi(0)$.  Choose $\delta$ such that $[x-2\delta, x+2\delta] \subset [-a,b]$ and $\omega(2\delta; \tau) < \epsilon$.  By hypothesis, $\tilde{\xi}(0)<x+\delta$, and so by Lemma \ref{Lemma:DriverTimes1} we have
    \begin{align*}
        \tau(x+\delta; \tilde{\xi}) \leq \tau(x+2\delta; \xi) < \epsilon,
    \end{align*}
    showing by monotonicity that $\tau(\tilde{x}; \tilde{\xi}) < \epsilon$ for all $\tilde{x} \in [\tilde{\xi}(0), x+\delta]$.  A parallel argument holds for $\tilde{x} \in [x-\delta, \tilde{\xi}(0)]$, and so we conclude that 
    \begin{align*}
        \tau(\tilde{x};\tilde{\xi}) = |\tau(\tilde{x};\tilde{\xi})- \tau(x;\xi)|<\epsilon
    \end{align*}
    whenever $|\tilde{x}-x|<\delta$ and $\| \tilde{\xi} - \xi\| < \delta$.

    In the case that $x \neq \xi(0)$, by symmetry we may assume $\xi(0) < x$, and we choose $\delta$ such that $[x-2\delta,x+2\delta] \subset (\xi(0), b]$ and $\omega(\delta;\tau) <\epsilon/2$.  Then if $|\tilde{x}-x|<\delta$ and $\|\tilde{\xi}-\xi\|<\delta$, we see from \eqref{Ineq:DriverTimesMonotone} that
    \begin{align*}
        |\tau(\tilde{x}; \tilde{\xi}) - \tau(\tilde{x}; \xi)| < \frac{\epsilon}{2},
    \end{align*}
    and thus 
        \begin{equation*}
        |\tau(\tilde{x};\tilde{\xi}) - \tau(x; \xi)| \leq |\tau(\tilde{x};\tilde{\xi}) - \tau(\tilde{x}; \xi)| + |\tau(\tilde{x};\xi) - \tau(x; \xi)| < \epsilon.\qedhere
    \end{equation*}

\end{proof}

\subsection{Continuity of $\xi \mapsto \tau^{-1}$}\label{Sec:DriverToTimesInverse}

The hitting times $\tau(\cdot;\xi)$ associated to $\xi \in S([0,T])$ are strictly monotonic on intervals on either side of $\xi(0)$ by Lemma \ref{Lemma:tauContinuous}.  So, as in \eqref{Eq:HittingTimesTwoFunctions}, we may consider them as two invertible functions $\tau_{\pm}$.  The following lemma says that the maps $\xi \mapsto \tau_+^{-1}$ and $\xi \mapsto \tau_-^{-1}$ are Lipschitz continuous with Lipschitz constant 1, where in each case both the domain and range are equipped with the sup norm on $[0,T]$.  

\begin{theorem}\label{Lemma:InverseTauLip}
    For $\xi, \tilde{\xi} \in S([0,T])$, $0< T \leq \infty$, we have
    \begin{align}\label{Ineq:InverseHittingLipschitz}
        \| \tau_\pm^{-1}(\cdot;\xi) - \tau_\pm^{-1}(\cdot;\tilde{\xi}) \|_{\infty[0,T]} \leq \| \xi - \tilde{\xi} \|_{\infty[0,T]}.
    \end{align}
    Furthermore, $C=1$ is the best-possible Lipschitz constant.
\end{theorem}
\noindent Note that the Lipschitz constant does not depend on $T$. Compare also Lemma \ref{Lemma:DriverToTimesNotUniform}, where we saw very different behavior for $\xi \mapsto \tau$.

\begin{proof}
    We show \eqref{Ineq:InverseHittingLipschitz} for $\tau_+^{-1}$; the argument for $\tau_-^{-1}$ is similar.  We start with two observations.  Setting $\|\xi - \tilde{\xi}\|_{\infty[0,T]} =:\delta$, we note by \eqref{Eq:LoewnerShift} and \eqref{Ineq:LoewnerOrdering} that
    \begin{align}\label{Ineq:hMapsDiff}
        h_t(y_0;\xi) + \delta = h_t(y_0 + \delta; \xi+\delta) \leq h_t(y_0+\delta; \tilde{\xi})
    \end{align}
    for all $t \leq T$ and $\xi(0)\leq y_0$ such that $t \leq \tau_+(y_0;\xi)$ (which implies $t\leq\tau_+(y_0+\delta; \tilde{\xi})$ by Lemma \ref{Lemma:DriverTimes1}).  We secondly observe that if for some driver $\eta$ we have $\eta(0) < y_1< y_1 + \epsilon < y_2$, then
    \begin{align}\label{Ineq:hMapsDiff2}
        h_t(y_1; \eta) + \epsilon < h_t(y_2; \eta)
    \end{align}
    whenever $t \leq \tau_+(y_1; \eta)$, as follows from the expansion of intervals in the upwards flow; recall \eqref{Eq:LoewnerIncrement}.  
    
    Now, fix $t_0 \in [0,T]$ and consider $y_0: = \tau_+^{-1}(t_0;\xi)$ and $\tilde{y}_0 := \tau_+^{-1}(t_0;\tilde{\xi})$, and suppose, without loss of generality, that $y_0 \leq \tilde{y}_0$.  We proceed by contradiction: if $y_0+\delta + \epsilon < \tilde{y}_0$ for some $\epsilon>0$, then by \eqref{Ineq:hMapsDiff},
    \begin{align*}
        h_{t_0}(y_0;\xi) +\delta +\epsilon \leq h_{t_0}(y_0+\delta; \tilde{\xi}) + \epsilon,
    \end{align*}
    while by \eqref{Ineq:hMapsDiff2} we have
    \begin{align*}
        h_{t_0}(y_0+\delta; \tilde{\xi}) + \epsilon < h_{t_0}(\tilde{y}_0, \tilde{\xi}),
    \end{align*}
    and thus combining these yields 
    \begin{align*}
         \|\xi - \tilde{\xi}\|_{\infty[0,T]} + \epsilon \leq h_{t_0}(\tilde{y}_0; \tilde{\xi})  - h_{t_0}(y_0; \xi) = \tilde{\xi}(t_0) - \xi(t_0), 
    \end{align*}
    a contradiction.  Hence $\tilde{y}_0 - y_0 \leq \delta$, as claimed.
    
    The fact that the Lipschitz constant is optimal is immediately evident from, say, constant drivers $\xi(t) \equiv 0$ and $\tilde{\xi}(t) \equiv \epsilon$.  Here 
    \begin{align*}
        \tau_+^{-1}(0;\tilde{\xi}) - \tau_+^{-1}(0;\xi) = \epsilon = \|\xi - \tilde{\xi} \|_{\infty}.
    \end{align*}
    Note that $C=1$ is still sharp under the more restrictive condition that $\xi, \tilde{\xi} \in S_0([0,T])$.  Indeed, consider $\xi_2(t) \equiv 0$ and $\tilde{\xi}_2(t) = t/\delta$ on $[0,\delta]$ for some small $\delta>0$, with $\tilde{\xi}_2(t) \equiv 1$ for $t \geq \delta$.  We see from \eqref{Eq:ZeroHittingTime} that $\tau_+^{-1}(\delta;\xi_2) = \delta^2/4$, while $\tau_+^{-1}(\delta;\tilde{\xi}_2) \geq \tilde{\xi}_2(\delta)=1$, and thus
    \begin{equation*}
        |\tau_+^{-1}(\delta;\xi)-\tau_+^{-1}(\delta;\tilde{\xi})| \geq 1- 2\sqrt{\delta} = ( 1- 2\sqrt{\delta}\,)\|\xi - \tilde{\xi} \|_{\infty[0,\delta]}.\qedhere
    \end{equation*}
\end{proof}

\subsection{Preliminary results on continuity properties of $\tau \mapsto \xi$}\label{Sec:TimesToDriver}

In this section, we begin to explore the question of whether the map $\tau \mapsto \xi$ can be defined, and if so, if it is continuous.  We show that for the zero driver $\tb{0}(t) \equiv 0$, if $\tau(\cdot; \xi) = \tau(\cdot; \tb{0})$, then $\xi = \tb{0}$, and thus $\tau \mapsto \xi$ is well defined at the zero.  Furthermore, if $\tau_n$ are hitting times for $\xi_n \in S_0([0,T])$ with $\tau_n \rightarrow \tau(\cdot; \tb{0})$ uniformly, then $\xi_n \rightarrow \tb{0}$ uniformly, and so $\tau \mapsto \xi$ is also continuous at $\tb{0}$.  See Theorem \ref{Thm:TimesToDriverContinuousZero} for precise statements.

Before moving to the proofs, we comment that moving from hitting times to drivers is more subtle than moving from drivers to hitting times.  One indication of this is that the proof of the above facts is, to our surprise, not trivial and seems to require some new machinery, and another indication is the breakdown of monotonicity properties that one has in the $\xi \mapsto \tau$ direction.  Indeed, given two drivers $\xi, \tilde{\xi} \in S_0([0,T])$, recall that if $\xi \leq \tilde{\xi}$, then, similar to \eqref{Ineq:LoewnerOrdering}, we have that the corresponding hitting-time functions satisfy 
\begin{align}\label{Ineq:HittingTimesMonotone}
    \tau_-\leq \tilde{\tau}_- \qquad \text{and} \qquad \tilde{\tau}_+ \leq \tau_+
\end{align}
on their common domains.  However, the converse implication does not hold: assuming the hitting times satisfy \eqref{Ineq:HittingTimesMonotone} does not imply $\xi \leq \tilde{\xi}$, as the following example shows.

\begin{example}
    Consider $\xi(t) \equiv 0$ and the driver $\tilde{\xi}$ which begins at 0 but then moves linearly with slope $1/\epsilon$ to reach value $1$ at time $\epsilon$.  Then at time 1, say, $\tilde{\xi}$ moves extremely fast to value $-\epsilon$ linearly with slope $-(1+\epsilon)/\epsilon$, then immediately back to value 1 with slope $(1+\epsilon)/\epsilon$.  We have $\xi, \tilde{\xi} \in S_0([0,1+2\epsilon])$ and the drivers are not monotone, but we claim that \eqref{Ineq:HittingTimesMonotone} still holds for $\epsilon$ sufficiently small.  
    
    Consider first the left inequality 
    \begin{align}\label{Ineq:HittingTimesMonotoneLeft}
        \tau_- \leq \tilde{\tau}_-.
    \end{align}
    This holds for all 
    \begin{align*}
        x\in [\tau_-^{-1}(1)\wedge \tilde{\tau}_-^{-1}(1),0] = [\tilde{\tau}_-^{-1}(1),0]
    \end{align*}
    since $\xi \leq \tilde{\xi}$ on $[0,1]$.  Set $x_0 := \tau_-^{-1}(1)$.  Since $\tilde{\xi}-\xi = 1$ on $[\epsilon,1]$, the flow $\tilde{x}$ of the same initial point $x_0$ under $\tilde{\xi}$ satisfies $\tilde{x}(1)<-\delta(\epsilon) <0$, where $\delta$ is increasing in $\epsilon$.  Thus, for small-enough $\epsilon$, we still have $\tilde{x}(1+2\epsilon) < \tilde{\xi}(1+2\epsilon)$, which shows that \eqref{Ineq:HittingTimesMonotoneLeft} holds on the common domain $[\tilde{\tau}_-^{-1}(1+2\epsilon),0]$ of $\tau_-$ and $\tilde{\tau}_-$.  
    
    That $\tau_+ \geq \tilde{\tau}_+$ on $[0,\tau_+^{-1}(1+2\epsilon)]$ is similarly clear: by \eqref{Eq:ZeroHittingTime} we have $\tau_+^{-1}(1+2\epsilon) = 2\sqrt{1+2\epsilon} = 2 + O(\epsilon)$, while $\tilde{\xi}$ has already eaten points past this on $[0,1]$ when $\epsilon$ is small, as $\tilde{\tau}_+^{-1}(1) \approx 3$.  
\end{example}

In short, to say much about the $\tau \mapsto \xi$ direction one has to find a mechanism for handling rapid oscillations in $\xi$.  The machinery we construct below works in the case of a constant driver $\tau(\cdot; \tb{C})$, which, without loss of generality, we may assume to be the zero driver $\tb{0}$.  Our principal tool is Lemma \ref{Lemma:FasterTimes}, which bounds the welding time $\tau$ for two points $x_0 <y_0$ in terms of how far $\xi(\tau)$ is from the initial average $(x_0+y_0)/2$.  We also find it helpful to write $\xi$ as a convex combination of the points it will weld, which leads to the following integral representation for the interval length $y(t)-x(t)$.
\begin{lemma}\label{Lemma:IntervalWidth}
    Suppose $\xi \in C([0,\tau])$ welds initial points $x_0 < \xi(0) < y_0$ at some time $\tau$.  Let $I(t) := y(t)-x(t)$ be the interval length at time $t$ and let $\alpha:[0,\tau) \rightarrow (0,1)$ be the convex combination coefficient yielding
    \begin{align}\label{Eq:XiConvexCombo}
        \xi(t) = (1-\alpha(t))x(t) + \alpha(t) y(t). 
    \end{align}
    Then 
    \begin{align}\label{Eq:IntervalWidth}
        I(t) = \sqrt{I(0)^2 -4\int_0^t \frac{ds}{\alpha(s)(1-\alpha(s))}}
    \end{align}
    for all $0 \leq t \leq \tau$.
\end{lemma}
\noindent Recall that, as usual, $x(t) = h_t(x_0;\xi)$ and $y(t) = h_t(y_0;\xi)$, with the $h_t$ maps satisfying \eqref{Eq:UpwardsLoewner}.  We note \eqref{Eq:IntervalWidth} is a generalization of the formula $t \mapsto \sqrt{y_0^2 -4t}$ for the image of a point $y_0 >0$ under the zero driver, as in this case 
\begin{align}\label{Eq:ZeroDriverInterval}
    I(t) = 2\sqrt{y_0^2-4t} = \sqrt{I_0^2 - 16t}.
\end{align}
\begin{proof}
    For $0 \leq t < \tau$, the Loewner equation \eqref{Eq:UpwardsLoewner} yields
    \begin{align*}
        \frac{d}{dt}\big(I(t)^2\big) = \frac{-4I(t)^2}{(y(t)-\xi(t))(\xi(t)-x(t))} = \frac{-4}{\alpha(t)(1-\alpha(t))},
    \end{align*}
    yielding \eqref{Eq:IntervalWidth}.  As $t \rightarrow \tau$, $I(t) \rightarrow I(\tau) = 0$, which shows the integral in the radical is convergent and that the formula also holds when $t=\tau$.
\end{proof}
Observe that when $t=\tau$, \eqref{Eq:IntervalWidth} says
\begin{align}\label{Eq:IntervalWidthEnd}
    I(0)^2 = \int_0^\tau \frac{4 ds}{\alpha(s)(1-\alpha(s))} \geq 16 \tau
\end{align}
by calculus on the function $x \mapsto \frac{4}{x(1-x)}$, thus showing 
\begin{align}\label{Ineq:MaxTime}
    \tau \leq \frac{(y_0-x_0)^2}{16}.
\end{align}
This maximum time is achieved by the driver which is constantly the average of $x_0$ and $y_0$ by \eqref{Eq:ZeroDriverInterval} (or \eqref{Eq:ZeroHittingTime}).  

The lemma which immediately yields our result on $\tau \mapsto \xi$ for the zero driver is the following.
\begin{lemma}\label{Lemma:FasterTimes}
    Suppose $\xi \in C([0,\tau])$ welds initial points $x_0 < \xi(0) < y_0$ satisfying $x_0 + y_0 = 0$ at time $\tau$ with $|\xi(\tau)| = \delta$.  Then $\tau \leq f(\delta)$ for some function $f$ which is strictly decreasing in $\delta$ and satisfies $f(0) = \frac{(y_0-x_0)^2}{16}$.
\end{lemma}
\noindent Note that, as the proof will make clear, we do not require that $\xi \in S([0,\tau])$, only that $\xi$ welds $x_0$ and $y_0$ at time $\tau$.  
\begin{proof}
    By symmetry it suffices to consider the case $\delta > 0$.  With $\alpha(t)$ as in \eqref{Eq:XiConvexCombo}, define $\alpha_0:= \frac{1}{2}+\frac{\delta}{4y_0}$ so that
    \begin{align*}
        (1-\alpha_0)x_0 + \alpha_0y_0 = \frac{\delta}{2},
    \end{align*}
    and consider the set
    \begin{align*}
        T_0 := \{ \, t \, : \, 1-\alpha_0 < \alpha(t) < \alpha_0 \,\},
    \end{align*}
    which is open and thus a finite or countable collection of open intervals.  For a constant $\epsilon_0 = \epsilon_0(\delta, y_0)>0$ to be determined below, we either have $(i)$ $|T_0^c| < \epsilon_0$ or $(ii)$ $|T_0^c| \geq \epsilon_0$, with $|T_0^c|$ the Lebesgue measure of $T_0^c$.  We proceed to explicitly construct a suitable bounding function for $\tau$ in each case.  
    
    As the intuition is easily obscured, it may be worth explicitly stating: the times $T_0$ represent the ``reasonable'' region where $\xi(t)$ is ``close'' to the average $(x(t)+y(t))/2$.  If $|T_0|$ is large, as in case $(i)$, then $\tau$ cannot be too large, or else $y$ would move past $\delta$ (as for the zero driver and $\tau$ approaching $(y_0-x_0)^2/16$).  If $|T_0|$ is ``small,'' as in $(ii)$, then the interval $I(t)$ is collapsing quickly since $\xi(t)$ is ``close'' to one of $x(t),y(t)$ often, which also makes the time ``small.''
    
    Beginning with case $(i)$, we estimate the movement of $y_0$ on $T_0$.  Since 
    \begin{align}\label{Ineq:IntervalWidthIntegralEst}
        \int_0^t \frac{4 ds}{\alpha(s)(1-\alpha(s))} \geq 16t,
    \end{align} 
    we see from \eqref{Eq:IntervalWidth} that $I(t) \leq \sqrt{I(0)^2 - 16t}$.  Thus, using the Loewner equation \eqref{Eq:UpwardsLoewner}, we find for $t \in T_0$ that
    \begin{align*}
        -\dot{y}(t) = \frac{2}{(1-\alpha(t))I(t)} \geq \frac{2}{\alpha_0\sqrt{I(0)^2 - 16t}},
    \end{align*}
    and that therefore that the cumulative change $\Delta y|_{T_0}$ of $y$ over times $t \in T_0$ satisfies
    \begin{align*}
        y_0 - \delta \geq - \Delta y|_{T_0} &\geq \int_{T_0}\frac{2dt}{\alpha_0\sqrt{I(0)^2 - 16t}}\\
        &\geq \int_0^{|T_0|}\frac{2dt}{\alpha_0\sqrt{I(0)^2 - 16t}}\\
        &= \frac{I(0)}{4\alpha_0} - \frac{\sqrt{I(0)^2-16|T_0|}}{4\alpha_0}.
    \end{align*}
    Here the first inequality is by the assumption that $y(\tau) = \delta$, while the second line follows by monotonicity of the integrand.  Solving for $|T_0|$ yields
    \begin{align}\label{Ineq:TimesFastBound1}
        \frac{I(0)^2}{16}- \frac{1}{16}\big( I(0)-4\alpha_0(y_0-\delta) \big)^2 \geq |T_0| > \tau - \epsilon_0,
    \end{align}
    and so if we choose 
    \begin{align}
        \epsilon_0 := \frac{1}{32}\big( I(0)-4\alpha_0(y_0-\delta) \big)^2= \frac{\delta^2(y_0+\delta)^2}{32y_0^2} >0, \label{Eq:epsilon0}
    \end{align}
    we see that \eqref{Ineq:TimesFastBound1} yields
    \begin{align*}
        \tau \leq \frac{(y_0-x_0)^2}{16}- \frac{\delta^2(y_0+\delta)^2}{32y_0^2} =: f_1(\delta),
    \end{align*}
    which is decreasing in $\delta$ and has the desired boundary value at $\delta=0$.
    
    We turn, then, to the second case $|T_0^c| \geq \epsilon_0$, with $\epsilon_0 = \epsilon_0(\delta)$ still given by \eqref{Eq:epsilon0}.  Using \eqref{Eq:IntervalWidthEnd}, the symmetry of $x \mapsto x(1-x)$ around $1/2$, and \eqref{Ineq:IntervalWidthIntegralEst}, we find that
    \begin{align*}
        \frac{(y_0-x_0)^2}{4} &= \int_{T_0^c}\frac{ds}{\alpha(s)(1-\alpha(s))} + \int_{T_0}\frac{ds}{\alpha(s)(1-\alpha(s))}\\
        &\geq \frac{|T_0^c|}{\alpha_0(1-\alpha_0)} + 4|T_0|\\
        &= \frac{\big( 1 -4 \alpha_0(1-\alpha_0) \big)|T_0^c|}{\alpha_0(1-\alpha_0)} + 4\tau\\
        &\geq \frac{\big( 1 -4 \alpha_0(1-\alpha_0) \big)\epsilon_0}{\alpha_0(1-\alpha_0)} + 4\tau
    \end{align*}
    since $1 - 4\alpha_0(1-\alpha_0) >0$ (recall $x_0$ and $y_0$ are symmetric about zero and $\delta/2 > 0$, and so $\alpha_0 \neq 1/2$).  We thus have
    \begin{align*}
        \tau \leq \frac{(y_0-x_0)^2}{16} - \Big(\frac{1}{4 \alpha_0(1-\alpha_0)} -1\Big)\epsilon_0(\delta) =: f_2(\delta),
    \end{align*}
    which again is decreasing in $\delta$ with initial value $f_2(0) =(y_0-x_0)^2/16$. Combining the cases, we thus see
    \begin{equation*}
        \tau \leq f_1(\delta) \wedge f_2(\delta) =: f(\delta).\qedhere
    \end{equation*}
\end{proof}
\noindent We can immediately conclude the following from Lemma \ref{Lemma:FasterTimes}.
\begin{theorem}\label{Thm:TimesToDriverContinuousZero}
    The map hitting-time-to-driver map $\tau \mapsto \xi$ is well defined and continuous at the zero driver $\tb{0}$.  More precisely, 
    \begin{enumerate}[$(i)$]
        \item If $\xi \in C([0,T])$ has hitting times $\tau(x;\xi) = x^2/4$ on some interval $[-y_0,y_0]$, then $\xi(t) \equiv 0$ on $0 \leq t \leq y_0^2/4$.
        \item If $\tau_n:[a_n,b_n] \rightarrow \mathbb{R}$ are hitting times for drivers $\xi_n \in S_0([0,T_n])$ with $a_n \rightarrow a$, $b_n \rightarrow -a$, and where $\tau_n$ converges uniformly to hitting time function $\tau_{\tb{0}}$ of the zero driver on compact subsets of $(a, -a)$, then $\xi_n$ converges uniformly to $\tb{0}$ on any $[0,T'] \subset [0,T)$, where $T = \tau_{\tb{0}}(a) = a^2/4$.
    \end{enumerate}
\end{theorem}
\begin{proof}
    $(i)$ Lemma \ref{Lemma:FasterTimes} implies that $\xi(y^2/4) = 0$ for each $0 \leq y \leq y_0$.
    
    $(ii)$  We use Arzela-Ascoli to show $\{\xi_n\}$ is precompact on any $[0,T'] \subset [0,T)$, and then show all subsequential limits are $\tb{0}$.
    
    Since $a_n < \xi_n(t) <b_n$ for all $0 \leq t \leq T$, the sequence $\{\xi_n\}$ is uniformly bounded.  If it is not equicontinuous on $[0,T'] \subset [0,T)$, then there exist a subsequence $\{\xi_{n_k}\}$, $\epsilon_1 >0$, and times $t_{n_k},t' \in [0,T']$ satisfying $t_{n_k} \rightarrow t'$, such that
    \begin{align}\label{Ineq:NotEquicont}
        |\xi_{n_k}(t_{n_k}) - \xi_{n_k}(t')| \geq \epsilon_1
    \end{align}
    for all $k$, as follows from negating the definition of equicontinuity and using the compactness of $[0,T']$.
    
    We note that $\tau_{n_k,\pm}^{-1} \xrightarrow{u} \tau_{z,\pm}^{-1}$ on $[0,T']$, as follows from choosing an appropriate compact of $(-a,a)$ and using lemmas \ref{Lemma:OurDini} and \ref{Lemma:InversesConverge}. Thus the points $x_{n_k}, y_{n_k}$ and $x_{n_k}', y_{n_k}'$ welded together by $\xi_{n_k}$ at times $t_{n_k}$ and $t'$, respectively, satisfy 
    \begin{align*}
        \max\{\,|x_{n_k} - x'|, |x_{n_k}'-x'|, |y_{n_k}-y'|, |y_{n_k}'-y'| \,\} \rightarrow 0
    \end{align*}
    as $k \rightarrow \infty$, where $x' := \tau_{\tb{0},-}^{-1}(t')$ and $y' := \tau_{\tb{0},+}^{-1}(t')$.  In particular, $x_{n_k} + y_{n_k} \rightarrow 0$ and  $x_{n_k}' + y_{n_k}' \rightarrow 0$.  Since by \eqref{Ineq:NotEquicont} we have either $|\xi_{n_k}(t_{n_k})| \geq \epsilon_1/2$ or $|\xi_{n_k}(t_{n_k}')| \geq \epsilon_1/2$, by Lemma \ref{Lemma:FasterTimes} there is some $\delta>0$ such that
    \begin{align*}
        \min\{\tau_{n_k}(x_{n_k}), \tau_{n_k}(x_{n_k}')\} < \tau(t') - \delta
    \end{align*}
    for all large $k$.  Since this contradicts $\tau_{n_k} \xrightarrow{u} \tau_{\tb{0}}$ on a sufficiently-large compact of $(-a,a)$, we conclude that the sequence $\{\xi_n\}$ is, indeed, equicontinuous on $[0,T']$.  
    
    Take any subsequential limit $\xi_{n_k} \rightarrow \tilde{\xi}$ on $[0,T']$.  We wish to show that the hitting time function $\tilde{\tau}$ for $\tilde{\xi}$ is the same as $\tau_{\tb{0}}$ and then apply part $(i)$ of the theorem. (Unfortunately we cannot jump to use Theorem \ref{Cor:TimesPointwise} because, \emph{a priori}, we do not know that $\tilde{\xi} \in S$.)  Fix $0 < y_0$ and $0 < \epsilon < y_0/2$.  Then by \eqref{Ineq:DriverTimesMonotone},
    \begin{align*}
        \tau_{n_k}(y_0-\epsilon) \leq \tilde{\tau}(y_0) \leq \tau_{n_k}(y_0+\epsilon)
    \end{align*}
    for all sufficiently-large $k$, and so in the limit we find
    \begin{align*}
        \frac{(y_0-\epsilon)^2}{4} \leq \tilde{\tau}(y_0) \leq \frac{(y_0+\epsilon)^2}{4}.
    \end{align*}
    As this holds for any sufficiently-small $\epsilon$, we conclude that $\tilde{\tau}(y_0) = y_0^2/4$, and similarly that $\tilde{\tau}(-y_0) = y_0^2/4$.  By assumption on $\tau_n$ and part $(i)$, we conclude $\tilde{\xi} = \tb{0}$ on $[0,T']$.  Hence all subsequential limits are the same, and $\xi \rightarrow \tb{0}$ uniformly on $[0,T']$.  
\end{proof}

\subsubsection{Positive evidence for the non-zero case}\label{Sec:PositiveEvidence}

In considering whether hitting time convergence implies driver convergence, we ran some numerical experiments to gain intuition.  We thought we may have generated a method to produce a counter-example, but the simulations actually produced convergent driving functions, thus yielding some positive evidence for a generalization of Theorem \ref{Thm:TimesToDriverContinuousZero}$(ii)$.  We share the construction and this evidence here.  

Let $\mathcal{T}_0$ be the collection of hitting-time functions $\tau(\cdot; \xi)$ generated by $\xi \in S_0([0,T])$, and choose $\tau(\cdot; \xi) \in \mathcal{T}_0$ corresponding to $\xi \neq \tb{0}$.  Say $\tau:[-a,b] \rightarrow \mathbb{R}$ with $\tau(a) = \tau(b) = T$.  At stage $n$ consider the pairs 
\begin{align*}
    (x_j,y_j) = \big( \tau_-^{-1}(jT/n; \xi ),\; \tau_+^{-1}(jT/n; \xi ) \big), \qquad j =1,\ldots, n,
\end{align*}
that $\xi$ welds together at times $jT/n$.  We construct a driver $\xi_n \in S_0$ that also welds $x_j$ to $y_j$ in time $jT/n$ but has ``large" oscillations, which will potentially destroy uniform converge to $\xi$.  

Indeed, start $\xi_n$ at zero and have it move linearly with extremely-large speed until it reaches $(x_1(\epsilon;\xi_n)+y_1(\epsilon;\xi_n))/2$ at some time $\epsilon$.  The true welding time for $(x_1,y_1)$ is
\begin{align*}
    \frac{T}{n} \leq \frac{(y_1-x_1)^2}{16}
\end{align*}
by \eqref{Ineq:MaxTime}, where the maximum is attained by the driver which is constantly $(y_1+x_1)/2$.  So if we set $\xi_n$ to be $(x_1(\epsilon;\xi_n)+y_1(\epsilon;\xi_n))/2$ for $\epsilon \leq t \leq \frac{T}{n}-\epsilon$, it does not weld these points, as $\epsilon$ is very small and $\xi$ is not generally constant.  At time $\frac{T}{n}-\epsilon$ we then use a large oscillation in $\xi_n$ to weld both points together in time $\epsilon$ (see below in the proof of Theorem \ref{Thm:WeldingToOthersNotContinuous} for a careful description on how to do this).  Thus $\xi_n$ welds the pair $(x_1,y_1)$ in exactly time $t_1=T/n$.  

We next have $\xi_n$ move in time $\epsilon$ to the average $\big(x_2(t_1+\epsilon; \xi_n)+y_2(t_1+\epsilon; \xi_n)\big)/2$ of the next pair.  Intuitively, since $\xi_1$ stayed ``far away'' from $x_1$ and $y_1$ until the large $\epsilon$-oscillation at the end, $x_2$ and $y_2$ have not moved as far as they normally would have under $\xi$ in $[0,T/n]$, and thus we expect
\begin{align*}
    \frac{T}{n} < \frac{\big( x_2(t_1; \xi_n)+y_2(t_1; \xi_n) \big)^2}{16}.
\end{align*}
So we have $\xi_n$ wait constantly at the average $(x_2(t_1+\epsilon; \xi_n)+y_2(t_1+\epsilon; \xi_n))/2$ on $\frac{T}{n} +\epsilon \leq t \leq \frac{2T}{n} - \epsilon$, and then use another large oscillation to weld the images of $x_2$ and $y_2$ together in time $\epsilon$.  Thus $\xi_n$ also welds $(x_2,y_2)$ in exactly the same amount of time as $\xi$.

We continue in this way to weld all the $(x_j,y_j)$ at times identical to $\tau(\cdot; \xi)$, which by the monotonicity of $\tau_\pm$ implies $\tau(\cdot;\xi_n) \xrightarrow{u} \tau(\cdot;\xi)$.  

We wondered if in the $\epsilon \rightarrow 0$ limit the oscillations of $\xi_n$ at the end of each interval $[jT/n,(j+1)T/n]$ would grow so large that $\|\xi_n - \xi\|_{\infty[0,T]}$ would be bounded below.  In numerical simulations this was not the case, however, suggesting that the convergence of hitting times is fairly robust.  We show a characteristic simulation in Figure \ref{Fig:Simulations}, and submit this as positive evidence for a generalization of Theorem \ref{Thm:TimesToDriverContinuousZero}$(ii)$.

We comment that our constructed driver $\xi_n$ is not the worst-case scenario, as it is not necessarily the one that minimizes the shrinking of the last interval $y_n-x_n$ on $0 \leq t \leq \frac{(n-1)T}{n}$ among all the drivers that weld the pairs $(x_j,y_j)$ at time $jT/n$ for $1 \leq j \leq n-1$.  If such an extremal driver could be approximated and simulated, one could then follow it with the above construction to weld $(x_n,y_n)$ on $\frac{n-1}{nT} \leq t \leq T$ and produce the largest-possible oscillation in $\xi_n$ for welding these last points.  The intuition is that the pair $(x_n,y_n)$ is ``far away'' when the driver is welding the other points, and so does not move very much; if we minimize its movement we then have the opportunity for an extremely large fluctuation in $\xi_n$.  It would be interesting to see if the oscillation would remain macroscopic in the $n \rightarrow \infty$ limit.

\begin{figure}
	\begin{center}
	\includegraphics[scale=0.7]{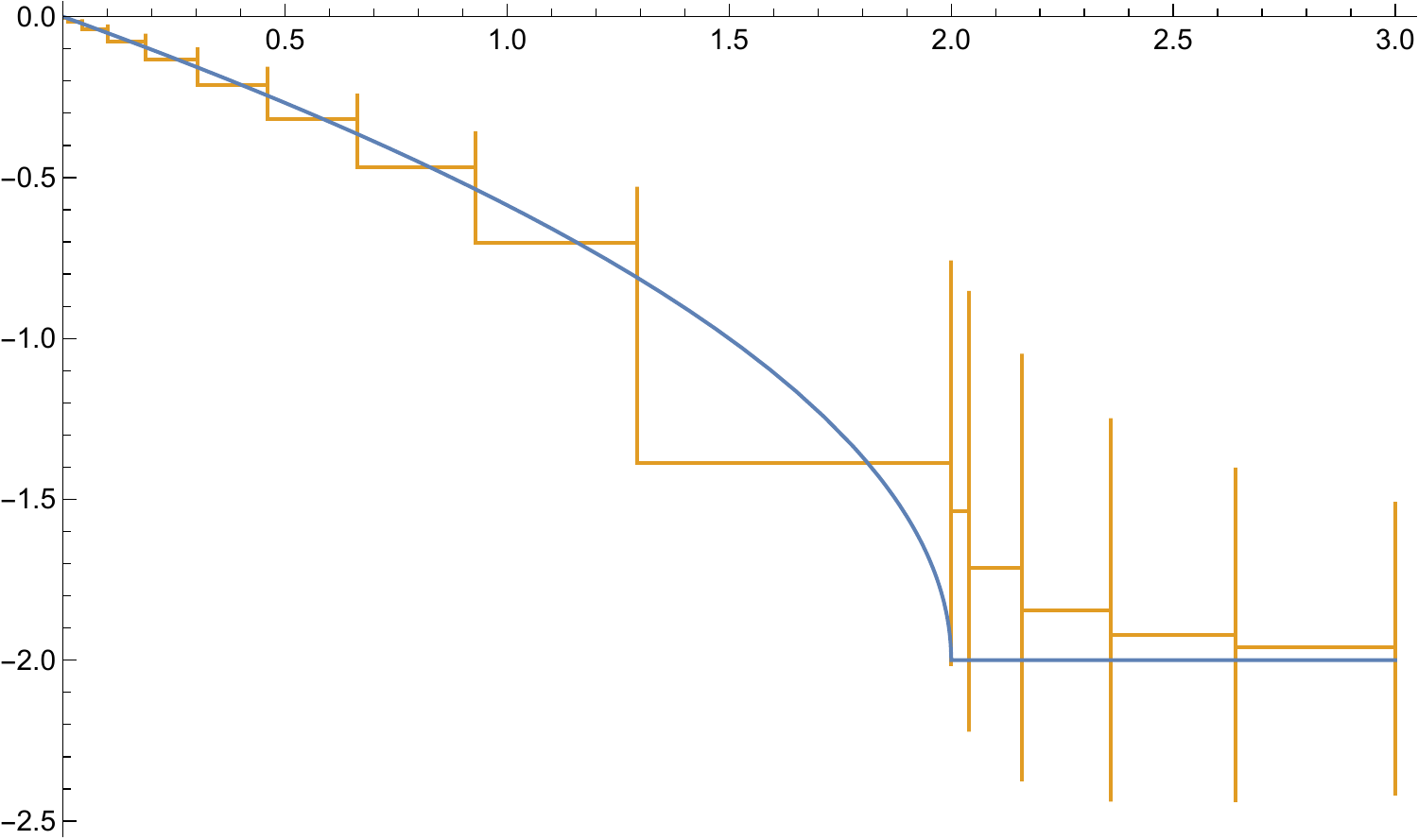}\\
	\bigskip
	\includegraphics[scale=0.7]{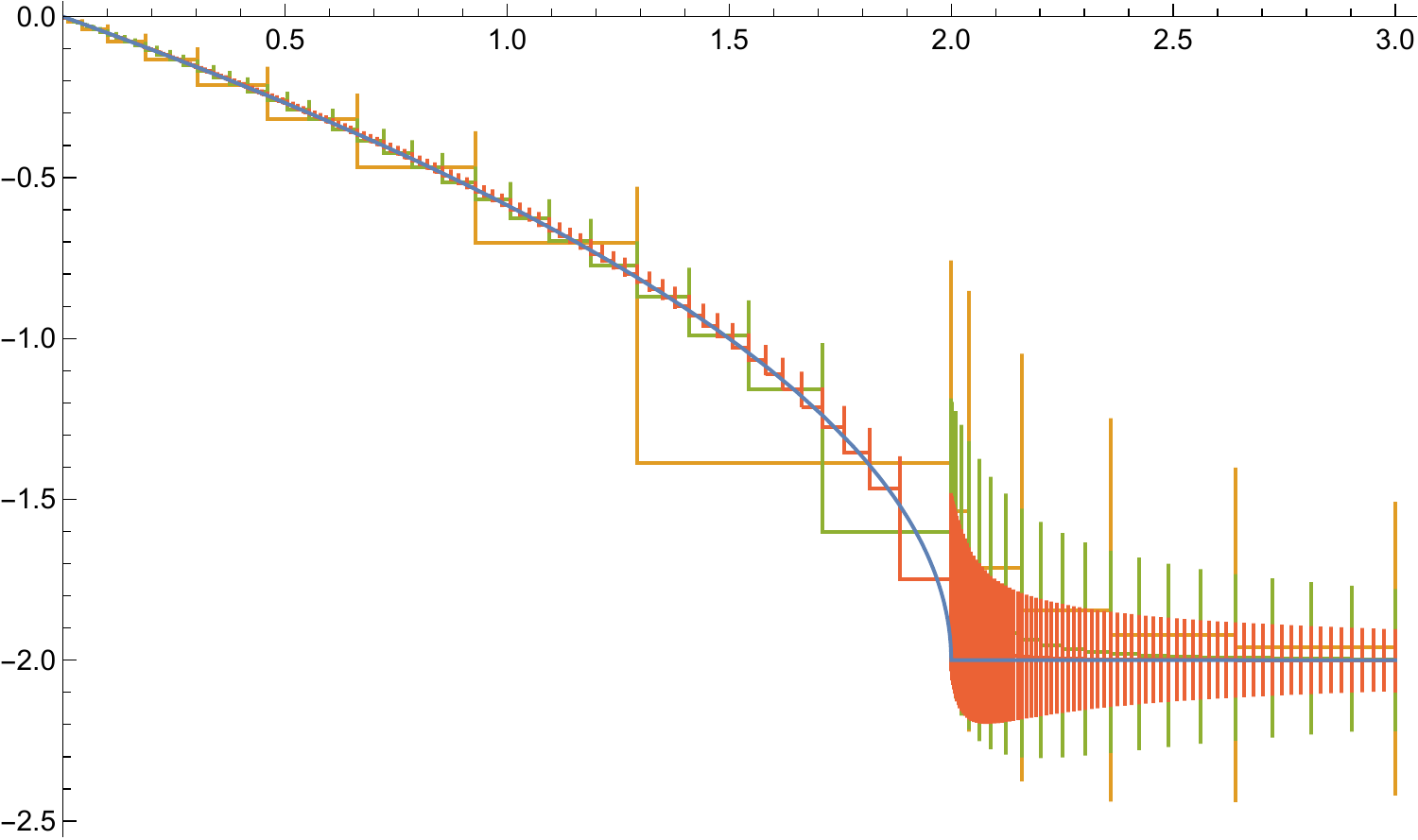}
    \end{center}
	\caption{\small Numerical simulations of drivers $\xi_n$ as constructed in \S\ref{Sec:PositiveEvidence}, in the $\epsilon \rightarrow 0$ limit (where the drivers move instantly and hence yield vertical lines).  Here the true driver $\xi$, pictured in blue, generates a line with angle $\pi/3$ to the positive reals on $0 \leq t \leq 2$, followed by a vertical line segment on $2 \leq t \leq 3$.  With a coarse approximation, the constructed driver $\xi_{n_1}$ in the upper figure struggles to stay close, but we see in the lower figure that the $\xi_n$ still converges as the mesh becomes finer.  Convergence at the corner is apparently only logarithmically fast in the mesh size, however (we quadruple the points in each iteration).  Given the difficulty at the non-smooth point of $\xi$, it is unclear whether the $\xi_n$ would converge for a rough fractal driver, such as, for instance, $\xi$ for the Von Koch snowflake \cite[Figure 2]{Lindrohdespace}.}
	\label{Fig:Simulations}
\end{figure}


\section{Continuity properties of $\tau \mapsto \varphi$ and $\xi \mapsto \varphi$}\label{Sec:Welding}


\subsection{Continuity of $\tau \mapsto \varphi$ and $\xi \mapsto \varphi$}\label{Sec:BlankToWelding}
For convenience in this section, we have our drivers start at zero, drawing them from $S_0$.  We also restrict to times $T<\infty$, as we will need domain compactness for uniform continuity.  

The continuity of $\tau \mapsto \varphi$ follows immediately from writing $\varphi = \tau_+^{-1} \circ \tau_-$ via Lemma \ref{Lemma:tauContinuous} and then using lemmas \ref{Lemma:OurDini} and \ref{Lemma:InversesConverge}; we leave the details to the interested reader.
\begin{lemma}\label{Lemma:TimesToWeldingConvergence}
    Let $0<T < \infty$ and let $\xi,\xi_n \in S_0([0,T])$.   Let $\tau:[-a,b]\rightarrow \mathbb{R}$ and  $\tau_n:[-a_n,b_n] \rightarrow \mathbb{R}$ be the hitting times for $\xi$ and $\xi_n$, respectively, and $\varphi:[-a,0] \rightarrow [0,b]$ and $\varphi_n:[-a_n,0] \rightarrow [0,b_n]$ their conformal weldings.  If $a_n \rightarrow a$, $b_n \rightarrow b$ and $\tau_n \xrightarrow{u} \tau$ on any $[c,d] \subset (-a,b)$, then $\varphi_n \xrightarrow{u} \varphi$ on $[-d,0]$ for any $[-d,0]\subset (-a,0]$.
\end{lemma}

We are more interested in the continuity of $\xi \mapsto \varphi$, the content of the following theorem.
\begin{theorem}\label{Thm:DriverToWeldingConvergence}
    Let $0<T < \infty$ and let $\xi,\xi_n \in S_0([0,T])$.   Let $\varphi:[-a,0] \rightarrow [0,b]$ and  $\varphi_n:[-a_n,0] \rightarrow [0,b_n]$ be the conformal weldings for $\xi$ and $\xi_n$, respectively.  If $\xi_n \xrightarrow{u} \xi$ on $[0,T]$, then $a_n \rightarrow a$, $b_n \rightarrow b$, and $\varphi_n \xrightarrow{u} \varphi$ on $[-c,0]$ for any $[-c,0]\subset (-a,0]$.
\end{theorem}
\begin{remark}
    As in Theorem \ref{Cor:TimesPointwise}, note that we also have
    \begin{align}\label{Ineq:EndpointsQuantWeldings}
        \max\{\, |a_n-a|, |b_n-b| \,\} \leq \|\xi_n-\xi\|_{\infty[0,T]}
    \end{align}
    by Theorem \ref{Lemma:InverseTauLip}.
\end{remark}

\begin{proof}
    Since $a_n \rightarrow a$ and $b_n \rightarrow b$ by \eqref{Ineq:EndpointsQuantWeldings}, $\varphi_n \in C_0([-c,0])$ for all large $n$.  Noting $\varphi = \tau_+^{-1}\circ \tau_-$ and writing $\tau_\pm$ and $\tau_{n,\pm}$ for the hitting-time functions generated by $\xi$ and $\xi_n$, respectively, we see for $-c \leq x \leq 0$ and sufficiently large $n$ that
    \begin{align}
        |\varphi_n(x) &- \varphi(x)|\notag \\
        &= |\tau_{n,+}^{-1}\big(\tau_{n,-}(x)\big)-\tau_+^{-1}\big(\tau_-(x)\big)| \notag\\
        &\leq |\tau_{n,+}^{-1}\big(\tau_{n,-}(x)\big)-\tau_+^{-1}\big(\tau_{n,-}(x)\big)| + |\tau_+^{-1}\big(\tau_{n,-}(x)\big) - \tau_+^{-1}\big(\tau_{-}(x)\big)| \notag\\
        &\leq \| \xi_n - \xi \|_{\infty[0,T]} + \omega\big(\|\tau_{n,-} - \tau_-\|_{\infty[-c,0]}; \tau_+^{-1}\big)\label{Ineq:AlmostQuantitative}
    \end{align}
    by Theorem \ref{Lemma:InverseTauLip}, where $\omega(\cdot; \tau_+^{-1})$ is the modulus of continuity of $\tau_+^{-1}$ on $[0,T]$, which exists by Lemma \ref{Lemma:tauContinuous}.  Since $\|\tau_{n,-} - \tau_-\|_{\infty[-c,0]} \rightarrow 0$ by Theorem \ref{Cor:TimesPointwise}, we have $\| \varphi_n - \varphi \|_{\infty[-c,0]} \rightarrow 0$.
\end{proof}

\begin{remark}
    Note that if one could control $\omega(\cdot; \tau_+^{-1})$ by the modulus of continuity $\omega(\cdot; \xi)$ of $\xi$ on $[0,T]$, then \eqref{Ineq:AlmostQuantitative} would yield a type of quantitative estimate.
\end{remark}

We saw in Theorem \ref{Cor:TimesPointwise} a sense in which $\xi \mapsto \tau$ is continuous, and in Lemma \ref{Lemma:DriverToTimesNotUniform} that it is not uniformly so.  The map $\xi \mapsto \varphi$ is analogous: it is continuous in the sense of Theorem \ref{Thm:DriverToWeldingConvergence} but there is again no universal modulus of continuity.

\begin{lemma}\label{Lemma:DriverToWeldingNotUniform}
There exist drivers $\xi_n, \tilde{\xi}_n \in S_0([0,T])$ such that $\|\xi_n - \tilde{\xi}_n\|_{\infty[0,T]} \rightarrow 0$ but where there exists $x$ welded by both $\xi_n$ and $\tilde{\xi}_n$ and $\epsilon>0$ such that   \begin{align}\label{Ineq:WeldingsFar}
    |\varphi_n(x)- \tilde{\varphi}_n(x)| \geq \epsilon
\end{align}
for all $n$, where $\varphi_n$ and $\tilde{\varphi}_n$ are the weldings for $\xi_n$ and $\tilde{\xi}_n$, respectively.
\end{lemma}

\begin{proof}
    We use the same drivers $\xi(\cdot, \delta)$ and $\tilde{\xi}(\cdot,\delta)$ as in the proof Lemma \ref{Lemma:DriverToTimesNotUniform} (including the modification in the last paragraph so that both are in $S_0$), and show that 
    \begin{align*}
        |\varphi^{-1}(y_0)- \tilde{\varphi}^{-1}(y_0)| \geq \epsilon,
    \end{align*}
    with $y_0$ the same point selected in that proof (to obtain \eqref{Ineq:WeldingsFar} reflect the drivers across the origin).  Consider the situation at time $t=\delta^2+\delta$, when $\xi$ and $\tilde{\xi}$ arrive at $-1$ and $-1-\delta$, respectively, for the first time.  We have that $y(\delta^2+\delta; \xi) = -1+2\delta$, and thus the point $x_0$ which will weld to it is, at that moment, at 
    \begin{align}\label{Eq:WeldingNotUniformx}
        x(\delta^2+\delta;\xi) = -1-2\delta.
    \end{align}
    By \eqref{Ineq:MovementOfy0}, 
    \begin{align*}
        y(\delta^2+\delta; \tilde{\xi}) - \tilde{\xi}(\delta^2+\delta) \geq 3\delta + \frac{1}{3},
    \end{align*}  
    and thus the point $\tilde{x}_0$ which welds to $y_0$ under $\tilde{\xi}$ satisfies
    \begin{align*}
        \tilde{x}(\delta^2+\delta; \tilde{\xi}) \leq -1-\delta - \Big( 3\delta + \frac{1}{3} \Big) = -4\delta - \frac{4}{3},
    \end{align*}
    and so by \eqref{Eq:WeldingNotUniformx} we see
    \begin{align}\label{Ineq:PreWeldFar}
        x(\delta^2+\delta; \xi) - \tilde{x}(\delta^2+\delta; \tilde{\xi}) \geq \frac{1}{3} + 2\delta.
    \end{align}
    As $\delta \rightarrow 0+$, simple estimates with the Loewner equation show $x(\delta^2+\delta;\xi) - x_0 \rightarrow 0$ and $\tilde{x}(\delta^2+\delta;\tilde{\xi}) - \tilde{x_0} \rightarrow 0$, and thus \eqref{Ineq:PreWeldFar} shows
    \begin{align*}
        x_0 -\tilde{x}_0 = \varphi^{-1}(y_0) - \tilde{\varphi}^{-1}(y_0) \geq \frac{1}{4}
    \end{align*}
    for all small $\delta$.
\end{proof}
As in Lemma \ref{Lemma:DriverToTimesJoint}, we can easily conclude from above results that $(x;\xi) \mapsto \varphi(x;\xi)$ is pointwise jointly continuous in $x$ and $\xi$.  This generalizes \cite[Thm. 1.2$(b)$]{Vlad}, as our statement covers all drivers $\xi$ generating simple curves, not just the a.s. Brownian motion case.

\begin{lemma}\label{Lemma:DriverToWeldingJoint}
    Let $0<T < \infty$ and $\xi \in S_0([0,T])$ a driver with welding $\varphi:[-a,0] \rightarrow [0,b]$. If $x \in (-a,0]$ and $\epsilon>0$, there exists $\delta = \delta(\epsilon, x,\xi)$ such that whenever $\tilde{x}\leq 0$ and $\tilde{\xi} \in S_0([0,T])$ satisfy 
    \begin{align*}
        \max\{\, |\tilde{x}-x|, \|\tilde{\xi} - \xi\|_{\infty [0,T]}  \,\} < \delta,
    \end{align*}
    then
    \begin{align*}
        |\varphi(\tilde{x};\tilde{\xi}) - \varphi(x; \xi)| < \epsilon.
    \end{align*}
\end{lemma}
\begin{proof}
    Write $\tilde{\varphi}$ and $\varphi$ for $\varphi(\cdot; \tilde{\xi})$ and $\varphi(\cdot; \xi)$, respectively, and similarly for their drivers and hitting times.  For $\frac{x-a}{2}\leq \tilde{x} \leq 0$, we have that $\tilde{\varphi}(\tilde{x})$ is defined whenever $\|\tilde{\xi}-\xi\|_{\infty[0,T]}$ is sufficiently small by Theorem \ref{Lemma:InverseTauLip}, and for such $\tilde{x}$, the triangle inequality yields
    \begin{align*}
        |\tilde{\varphi}(\tilde{x}) - \varphi(x)| &\leq \|\tilde{\varphi} - \varphi\|_{\infty[(x-a)/2,0]} + \omega(|\tilde{x}-x|; \varphi)\\
        &\leq \| \tilde{\xi} - \xi \|_{\infty[0,T]} + \omega\big(\|\tilde{\tau} - \tau\|_{\infty[(x-a)/2,0]}; \tau_+^{-1}\big) + \omega(|\tilde{x}-x|; \varphi)
    \end{align*}
    by \eqref{Ineq:AlmostQuantitative}, where $\omega(\cdot; \varphi)$ is the modulus of continuity of $\varphi$ on $[-a,0]$, and $\omega(\cdot; \tau_+^{-1})$ that for $\tau_+^{-1}$ on $[0,T]$.  As $\xi$ and $\varphi$ are fixed, $\tau_+^{-1}$ is determined by $\xi$, and the hitting times are continuous by Theorem \ref{Cor:TimesPointwise}, we may choose $\delta$ small enough such that each of the three terms is less than $\epsilon/3$.
\end{proof}

\subsection{$\varphi \mapsto \tau$ and $\varphi \mapsto \xi$ are not continuous}\label{Sec:WeldingCounterexample}
\begin{figure}[h]
	\begin{center}
	\includegraphics[width=0.9\textwidth]{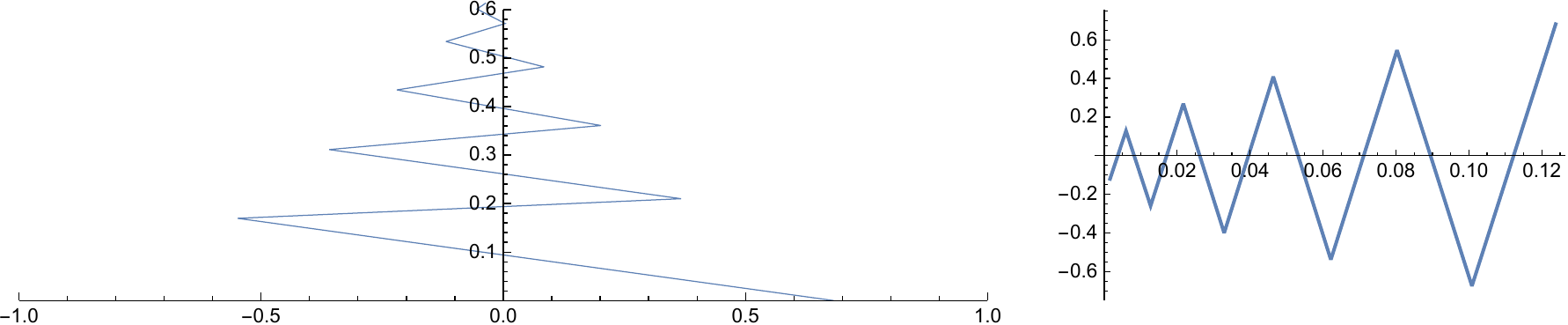}\\
	\bigskip
	\includegraphics[width=0.9\textwidth]{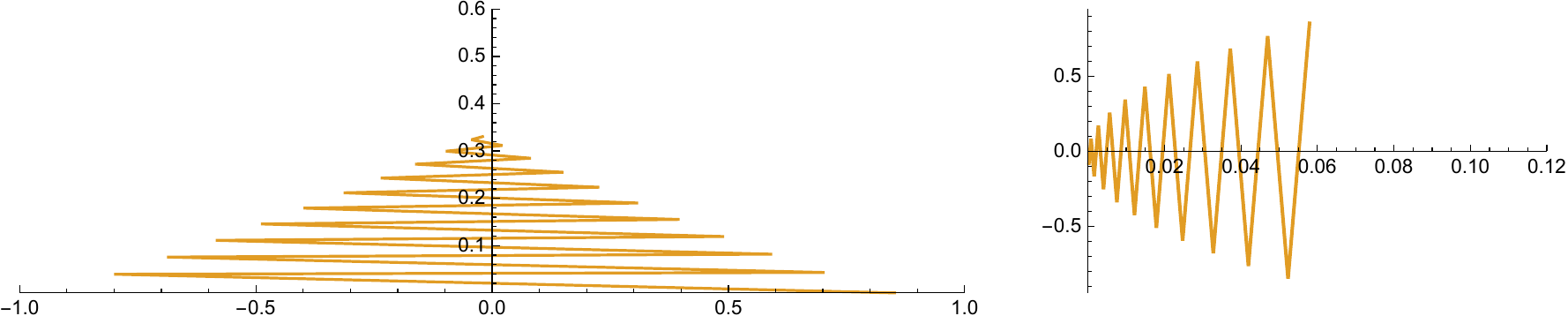}\\
	\bigskip
	\includegraphics[width=0.9\textwidth]{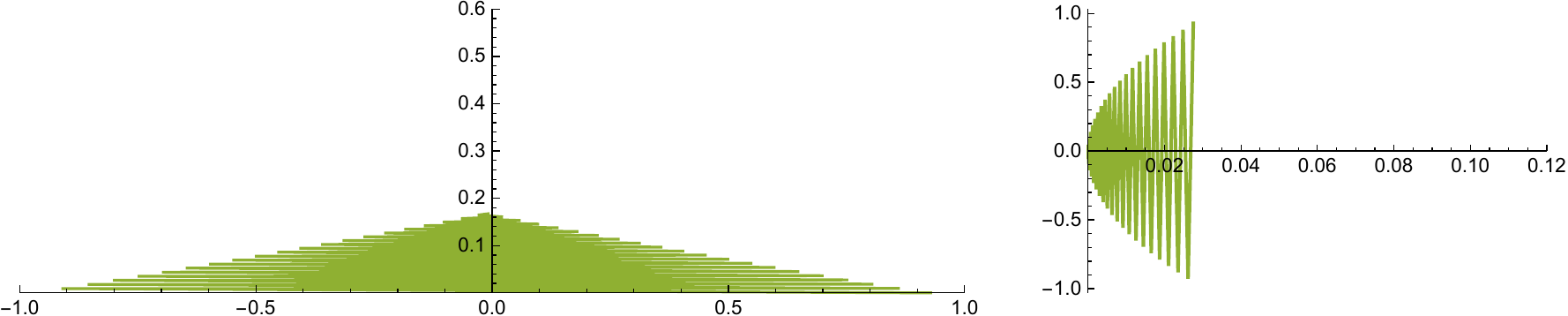}
    \end{center}
	\caption{\small Numerical simulations of curves $\gamma_n$ (left) and the drivers $\xi_n$ (right) which give a counterexample to the converse of Theorem \ref{Thm:DriverToWeldingConvergence}.  Here the weldings $\varphi_n$ of the $\gamma_n$ converge uniformly on $[-1,0]$ to $\varphi(x)=-x$, but the drivers stay far in supremum distance from the zero driver. (Note we have simplified the $\gamma_n$ by drawing them as piecewise polygonal, which would not quite be the case.)}
	\label{Fig:CounterExample}
\end{figure}
\begin{theorem}\label{Thm:WeldingToOthersNotContinuous}
    Let $\varphi(x) =-x$ on $[-1,0]$ be the welding for the vertical line segment $[0,i]$, with corresponding driver $\tb{0}$ and hitting times $\tau_{\tb{0}}$.  There exist $\varphi_n$ corresponding to simple curves $\gamma_n$ with drivers $\xi_n$ and hitting times $\tau_n$ such that $\|\varphi_n - \varphi\|_{\infty[0,1]} \rightarrow 0$ but where $\xi_n \not\rightarrow \tb{0}$ and $\tau_n \not\rightarrow \tau_{\tb{0}}$.
\end{theorem}
\begin{proof}
We build $\gamma_n$ by a piece-wise linear driving function $\xi_n$ which welds each $k/n$ to $-k/n$, $k=1,\ldots, n$, under its upwards Loewner flow.  The intuition is that $\xi_n$ will capture the initial points extremely fast through rapid oscillations, and hence the latter points will not have time to move very far, affording $\xi_n$ the opportunity to travel far from 0 to capture them.  By construction, the welding $\varphi_n$ generated by $\xi_n$ will satisfy $\varphi_n(-k/n) = k/n$ for all $k \in \{1, \ldots, n\}$, and since the weldings are monotone, this yields $\varphi_n \xrightarrow{u} \varphi$ on $[-1,0]$.  

We first note that we can weld $1/n$ to $-1/n$ in an arbitrarily-small amount of time.   Write $x_1 = -1/n$ and $y_1=1/n$, with $x_1(t)$ and $y_1(t)$ their images after time $t$ in the upwards Loewner flow generated by $\tilde{\xi}_1$, which we now construct.  Indeed, starting from $\tilde{\xi}_1(0)=0$, move $\tilde{\xi}_1$ linearly to $y_1(\epsilon_1)-\epsilon_1$ in time $\epsilon_1$, for some small $\epsilon_1 >0$.  Then $\dot{y}_1(\epsilon_1) = -2/\epsilon_1$, and have $\tilde{\xi}_1$ rush back towards $x_1(t)$ at precisely this same speed, until the time $t_1$ when it is exactly half-way between $x_1(t_1)$ and $y_1(t_1)$.  Then we freeze $\tilde{\xi}_1$ at this point $\tilde{\xi}_1(t_1)$ and let $x_1$ and $y_1$ flow together until they weld at time $T_1$.

We have $\tilde{\xi}_1$ is piecewise linear thus an element of $S_0([0,T_1])$, and we note $T_1$ is small: when moving back towards $x_1$, the distance $\tilde{\xi}_1$ travels is less than $2/n$, and so the time required is less than $\epsilon_1/n$.  Once $\tilde{\xi}_1$ stops, the image of $y_1$ is still $\epsilon_1$ away, and so takes $\epsilon_1^2/4$ units of time to reach $\tilde{\xi}_1(t_1)$.  Thus 
\begin{align*}
    T_1 < \epsilon_1 + \epsilon_1/n + \epsilon_1^2/4.
\end{align*}

After flowing up with such a driver to weld the pair $(x_1,y_1)$, we can repeat the idea to capture subsequent points with $\tilde{\xi}_j$'s, and can thus generate a $\xi_n$ that welds the first $n-1$ pairs in $T_{n-1}<1/n^2$ time.  The last mesh point remaining on the right is the image $y_n(T_{n-1})$ of $y_n(0)=1$, and since for all $0 \leq s \leq T_{n-1}$ we have the very coarse estimate
\begin{align*}
    y_n(s) - \xi(s) \geq y_n(s) - y_{n-1}(s) \geq y_n(0) - y_{n-1}(0) = \frac{1}{n}
\end{align*}
by \eqref{Eq:LoewnerIncrement}, we see $y_n$ has moved towards $\xi_n$ no more than $2n(1/n^2) = 2/n$ units, showing
\begin{align*}
    \sup |\xi_n| > 1 - \frac{3}{n}
\end{align*}
when $\xi_n$ moves fast enough after $T_{n-1}$ towards $y_n$.  

We conclude $\varphi_n \xrightarrow{u} \varphi$ but $\xi_n \not\rightarrow \tb{0}$.  Furthermore, we may build $\xi_n$ to weld all the points in time $2/n^2$, showing the hitting times also do not converge to $\tau_{\tb{0}}$. 
\end{proof}
Figure \ref{Fig:CounterExample} gives a numerical approximation for the curves $\gamma_n$ generated by the $\xi_n$, which collapse to the real interval $[-1,1]$ as $n \rightarrow \infty$.

\subsection{An application: convergence of a zipper-like algorithm for welding using minimal-energy curves}\label{Sec:Zipper}

Let $\gamma([0,T])$ be a finite simple curve in $\mathbb{H}\cup\{0\}$ with associated upwards driver $\xi$ and conformal welding $\varphi:[-a,0] \rightarrow [0,b]$.  Recall that $\gamma$ has finite \emph{Loewner energy} if the Dirichlet energy of $\xi$ is finite, i.e. $\xi$ is absolutely continuous and
\begin{align}\label{Eq:LoewnerEnergy}
    I_L(\xi) := \frac{1}{2}\int_0^T \dot{\xi}(t)^2dt < \infty.
\end{align}
The Loewner energy was introduced in \cite{ShekharFriz} and subsequently saw rapid development in \cite{RohdeWang,YilinFredwelding,YilinFredKufarev,WangReverse,WangEquiv}, to give an incomplete list.  In short, it has fascinating connections to a diverse array of fields: probability theory, complex analysis, hyperbolic geometry and geometric measure theory \cite{Bishop}, and even \Teich theory.  See \cite{Yilinsurvey} for a helpful overview.  

We use Loewner energy minimizers in this section to address a conformal welding approximation question.  Given a partition $\mathcal{P} = \{(x_j,y_j)\}_{j=1}^N$ of $[-a,b]$,
\begin{align}\label{Ineq:Partition}
    -a = x_N < x_{N-1} < \cdots < x_1 < x_0=0=y_0 < y_1 < \cdots < y_N = b
\end{align}
with $\varphi(x_j) = y_j$ for each $j$, an interesting question is when a curve $\tilde{\gamma}$ which welds each pair $(x_j,y_j)$ together is close to $\gamma$.  We have seen in \S\ref{Sec:WeldingCounterexample} that this is not always the case.  Indeed, given $\varphi \mapsto \xi$ and $\xi \mapsto \gamma$ are both not continuous, we expect $\varphi \mapsto \gamma$ to exhibit a number of pathologies.  

This question of closeness of $\gamma$ given closeness of $\varphi$ is related to the \emph{(domain) zipper algorithm} of Don Marshall \cite{Marshallrohde}, which seeks to approximate a conformal map $f$ to a domain $\Omega$ via a map $f_n$ which maps to a domain $\Omega_n$ whose boundary agrees with  $\partial \Omega$ on a given mesh/discretization $\mathcal{Q} \subset \partial \Omega$.\footnote{Thus note all ``zippers'' in this section are distinct from Sheffield's quantum zipper in probability theory.}  The map $f_n$ is built from composing $\#\mathcal{Q}$ conformal maps, where each subsequent map, in effect, draws a boundary segment between the next two points in $\mathcal{Q}$.  The question here is, when $\mathcal{Q}$ is very fine (and one ``draws'' a reasonable arc between successive points), is $\partial\Omega_n$ actually uniformly close to $\partial\Omega$?  There is a similar algorithm for weldings, Marshall's \emph{welding zipper algorithm}, which seeks to reconstruct $\gamma$ through composing $N$ conformal maps which ``zip up'' the partition \eqref{Ineq:Partition} discretizing $\varphi$ one pair at a time, producing a curve $\gamma_n$ whose welding $\varphi_n$ agrees with $\varphi$ at each $x_j$.  The question for the welding zipper is: what conformal maps can you use for each zip to guarantee that $\gamma_n$ is close to $\gamma$?

Both versions of the zipper, it turns out, work remarkably well in practice and have become something of industry standards for numerically computing conformal maps.  Proving convergence, however, has been elusive. For the domain zipper, the only proof is for when one draws hyperbolic geodesic segments between subsequent boundary points \cite{Marshallrohde}, and convergence for the welding zipper remains open (though see \cite{Mesikepp} for a partial result and further discussion). 

In this subsection, we give as a corollary of Theorem \ref{Thm:DriverToWeldingConvergence} a positive convergence result to an algorithm \emph{similar} to the welding zipper.  We construct curves $\gamma_n$ whose weldings match $\varphi$ on $\mathcal{P}_n$, but we create each $\gamma_n$ ``all at once'' through minimizing Loewner energy among all such curves, instead of building it through $N$ compositions.\footnote{We do not propose concrete means to actually compute the minimizers $\gamma_n$, and so we are admittedly using the term ``algorithm'' rather loosely.  Our point is to allude to the welding zipper algorithm, our source of inspiration.}  That is, whereas the zipper welds the first two points $x_1,y_1$ with some map $F_1$, and then welds the images $F_1(x_2),F_1(y_2)$ of the next two points under $F_1$ with some $F_2$, and so on, solving the welding problem with $F_N\circ \cdots \circ F_1$, we start with curves which already weld all pairs in $\mathcal{P}$ and minimize energy among them.  This makes the problem more tractable, and we only need existing results once Theorem \ref{Cor:TimesPointwise} establishes existence of minimizers.

Let us write $|\mathcal{P}| := \max\{\, x_{j-1}-x_j, y_j-y_{j-1} \,\}_{j=1}^N$ for the norm of the partition and say $\varphi$ \emph{welds} $\mathcal{P}$ if $\varphi(x_j) = y_j$ for all $1 \leq j \leq N$.

\begin{theorem}\label{Cor:Zipper}
    Let $\gamma:[0,T] \rightarrow \mathbb{H} \cup \{x\}$ be a finite curve of finite Loewner energy, with upwards driver $\xi \in S_0([0,T])$ and welding $\varphi:[-a,0] \rightarrow [0,b]$.  
    \begin{enumerate}[$(i)$]
        \item\label{Thm:MinimizersExist} For any partition $\mathcal{P}$ of $[-a,b]$ as in \eqref{Ineq:Partition}, there exists a curve $\gamma_{\mathcal{P}}$ with driver $\xi_{\mathcal{P}} \in S_0$ which minimizes the Loewner energy among all curves welding $\mathcal{P}$.  
        \item\label{Thm:MinimizersConverge} If $\{\mathcal{P}_n\}$ is any sequence partitions with $|\mathcal{P}_n| \rightarrow 0$, and $\{\gamma_n\}$ a sequence of corresponding  Loewner-energy minimizers from $(i)$, then in the half-plane capacity parametrizations of the curves, 
    \begin{align}\label{Lim:CurvesConverge}
        \lim_{n \rightarrow \infty} \| \gamma_n - \gamma\|_{\infty[0,T']} =0
    \end{align}
    for any $[0,T'] \subset [0,T)$.  Furthermore, the Loewner energies of the entire curves satisfy
    \begin{align}\label{Lim:EnergiesIncrease}
        \lim_{n \rightarrow \infty} I_L(\gamma_n) = I_L(\gamma).
    \end{align}
    If the partitions are nested, $\mathcal{P}_{n} \subset \mathcal{P}_{n+1}$ for all $n$, this limit is non-decreasing.
    \end{enumerate}
\end{theorem}
\noindent We note that the reason we must restrict to $[0,T'] \subset [0,T]$ in \eqref{Lim:CurvesConverge} is that, \emph{a priori}, hcap$(\gamma_n)$ could be less than hcap$(\gamma)$ for all $n$.\footnote{Recall the similar technicality due to the changing domains of the $\tau_n$ discussed in Example \ref{Eg:NoEndpoints}.}  The proof will show that if we rescale all the $\gamma_n$'s to have the ``correct'' time $T$ via setting $\tilde{\gamma} := \sqrt{\frac{T}{T_n}}\gamma_n$, then
\begin{align}\label{Lim:CurvesConverge2}
    \|\tilde{\gamma}_n - \gamma\|_{\infty[0,T]} \rightarrow 0
\end{align}
in the half-plane-capacity parametrizations.  In fact, our strategy to prove \eqref{Lim:CurvesConverge} will be to first show \eqref{Lim:CurvesConverge2}.

Note also that we normalize so that $\xi(0)=0$ (corresponding to the conformal welding exchanging intervals on either side of the origin).  Considering $\gamma$ as generated by $\xi$ on $[0,T]$, we thus have $\gamma(0) = x=\xi(T)$, which is not necessarily zero.  

We precede the proof by collecting several known results that we will use.  
\begin{proposition}[Lemma 4.2 \cite{LMR}]\label{Prop:LMR}
    Let $\gamma_n, \tilde{\gamma}:[0,T] \rightarrow \mathbb{H} \cup \{0\}$ be simple curves parametrized by half-plane capacity, with $\lambda_n \in S_0([0,T])$ the downwards driving functions for $\gamma_n$.  If $\|\gamma_n - \tilde{\gamma}\|_{\infty[0,T]} \rightarrow 0$ and there exists $\lambda \in C_0([0,T])$ such that  $\|\lambda_n - \lambda\|_{\infty[0,T]} \rightarrow 0$, then $\tilde{\gamma} = \gamma^\lambda$.  That is, $\tilde{\gamma}$ is the curve generated by $\lambda$.
\end{proposition}

\begin{proposition}[Prop. 2.1$(iii)$ \cite{ShekharFriz}]\label{Prop:Holder1/2}
    If $I_L(\gamma) \leq M$, the half-plane capacity parametrization $\gamma$ is H\"{o}lder-$1/2$ with \Hol semi-norm $|\gamma|_{1/2} \leq Ce^{CM}$ for some $C>0$.
\end{proposition}
\noindent Note that this result also follows, although without the explicit bound on $|\gamma|_{1/2}$, from \cite[proof of Lemma 4.1]{LMR} and the fact that finite-energy curves have locally-small \Hol norm.

We also need a continuity property of half-plane capacity.  For a set $F \subset \mathbb{H}$, the \emph{(closed) $\epsilon$-neighborhood of $F$ in $\mathbb{H}$} is
\begin{align}\label{Def:EpsilonNeighborhood}
    F^\epsilon := \mathbb{H} \cap \bigcup_{z \in F} \overline{B_\epsilon(z)},
\end{align}
and if $F$ is bounded and relatively closed in $\mathbb{H}$, $\fil(F)$ is the complement of the unbounded connected component of $F^c$.
\begin{proposition}[Lemma 4.4 \cite{Kemp}]\label{Prop:hcapUniform1}
    There are positive constants $\alpha$ and $C=C(R)$ such that if a compact $\mathbb{H}$-hull $K$ satisfies $K \subset \fil(K^\epsilon) \subset B_R(x)$ for some $x \in \mathbb{R}$, then
    \begin{align}\label{Ineq:hcapHolder}
        \hcap(K^\epsilon) - \hcap(K) \leq C\epsilon^\alpha.
    \end{align}
\end{proposition}
\noindent The actual statement in \cite{Kemp} is slightly stronger, using $\hcap(\fil(K^\epsilon))$ instead of $\hcap(K^\epsilon)$, but \eqref{Ineq:hcapHolder} suffices for our purposes (while $K^\epsilon$ may not have simply-connected complement, its half-plane capacity is still well defined through the probabilistic definition of $\hcap$ in  \eqref{Def:hCapProbability}.) Recall that the Hausdorff distance between closed sets $E,F$ is 
\begin{align*}
    d_h(E,F) := \inf\Big\{\, \epsilon >0 \; | \; E \subset \bigcup_{z \in F} \overline{B_\epsilon(z)} \quad \text{and} \quad F \subset \bigcup_{z \in E} \overline{B_\epsilon(z)} \,\Big\}.
\end{align*}
Proposition \ref{Prop:hcapUniform1} immediately yields the following.
\begin{corollary}\label{Prop:hcapHolder}
    Fix $R>0$.  The half-plane capacity, as a map from the compact $\mathbb{H}$-hulls inside $\overline{B_R(0)}$ equipped with $d_h(\cdot,\cdot)$ to the reals, is \Hol continuous.
\end{corollary}
\begin{proof}
    For compact $\mathbb{H}$-hulls $K_1, K_2 \subset \overline{B_R(0)}$, write $d:= d_h(K_1,K_2) \leq 2R$.  Using the superscript notation of  \eqref{Def:EpsilonNeighborhood}, we see by $\hcap$  monotonicity and \eqref{Ineq:hcapHolder} that
    \begin{align*}
        |\hcap(K_1) - \hcap(K_2)| &\leq |\hcap(K_1) - \hcap(K_1^d)| + |\hcap(K_2^{2d})-\hcap(K_2)|\\
        &\leq (C +C2^\alpha)d^\alpha ,
    \end{align*}
    where $C=C(4R)$ from Proposition \ref{Prop:hcapUniform1}.
\end{proof}
\noindent We immediately obtain the following, which is the form of the continuity result we will use.
\begin{corollary}\label{Cor:hcapContinuousCurves}
    Let $\xi_n \in S([0,T_n])$ and $\xi \in S([0,T])$ be drivers generating half-plane-capacity parametrized curves $\gamma_n:[0,T_n] \rightarrow \mathbb{H} \cup \{\xi_n(T_n)\}$ and $\gamma:[0,T] \rightarrow \mathbb{H} \cup \{\xi(T)\}$, respectively.  If $\gamma_n$ converges uniformly to $\gamma$ in some parametrization, which is to say
    \begin{align*}
        \| \gamma_n \circ \sigma_n - \gamma \circ \sigma \|_{\infty[0,A]} \rightarrow 0
    \end{align*}
    for some increasing, continuous functions $\sigma_n: [0,A] \rightarrow [0,T_n]$, $\sigma:[0,A] \rightarrow [0,T]$, then $\hcap(\gamma_n) \rightarrow \hcap(\gamma)$.
\end{corollary}

We proceed with the proof of Theorem \ref{Cor:Zipper}.  In \cite[Lemma 4.1]{MesikeppEnergy}, the second author proved the existence of Loewner-energy minimizers for a single pair $(x_1,y_1)$, also using Theorem \ref{Thm:DriverToWeldingConvergence}.  We extend the idea for the existence of $\gamma_{\mathcal{P}}$, repeating some details for the convenience of the reader.
\begin{proof}
    $(i)$ We first observe that the set
    \begin{align*}
        D(\mathcal{P},C) := \{\, \xi \in S_0 \; : \; I_L(\xi) \leq C, \gamma^\xi \text{ welds }\mathcal{P} \,\}
    \end{align*}
    is non-empty for sufficiently-large $C$.  This is because one can repeatedly map up with conformal maps to the complement of circular arc segments orthogonal to $\mathbb{R}$, welding two points at a time to the base of the arc, to obtain a simple curve with $\tilde{\varphi}$ which welds $\mathcal{P}$, and driven by some $\tilde{\xi}$.  Each circular arc segment has finite energy, and so $I_L(\tilde{\xi}) <\infty$ because there are finitely-many pairs $(x_j,y_j)$.\footnote{This iterative construction is an example of the welding zipper algorithm.  See \cite[Lemma 5.7]{Mesikepp} for details on the finiteness of energy of the circular arcs, such as an explicit energy formula.}  Next, take a sequence $\xi_n \in D_0 := D(\mathcal{P}, I_L(\tilde{\xi}))$, such that 
    \begin{align*}
        \lim_{n \rightarrow \infty} I_L(\xi_n) = \inf\{\, I_L(\eta) \; : \; \eta \in D_0 \,\}.
    \end{align*}
    We claim we may suppose all the $\xi_n$ are defined on a universal interval $[0,T_U]$ of capacity time.  Indeed, since the diameter of any curve $\tilde{\gamma}$ welding $-a$ to $b$ is comparable to $b+a$ \cite[top of p.74]{Lawler}, and
    \begin{multline*}
        \hcap(\tilde{\gamma}) \leq \hcap\big( \diam(\tilde{\gamma})\overline{B_1(0)}\cap \mathbb{H} \big)\\ \leq C^2(a+b)^2\hcap\big(\overline{B_1(0)} \cap \mathbb{H} \big) \leq C^2(a+b)^2
    \end{multline*}
    by scaling and monotonicity of $\hcap(\cdot)$ and \eqref{Eq:Diskhcap}, the times $T_n$ for $\xi_n$ to weld $-a$ to $b$ are all bounded.  Furthermore, we can extend any $\xi_n$ past $T_n$ by the constant function $\xi_n(T_n)$ without adding energy, and thus some such $T_U$ exists.
    
    Since $\{I_L(\xi_n)\}$ is bounded, by \eqref{Eq:LoewnerEnergy} and H\"{o}lder's inequality, $\{\xi_n\}$ is bounded and equicontinuous on $[0,T_U]$ and thus precompact.  If $\xi_{n_k} \rightarrow \xi'$ is any subsequential uniform limit, by the lower-semicontinuity of the energy in this topology \cite[\S2.2]{WangReverse},
    \begin{align*}
        I_L(\xi') \leq \liminf_{k \rightarrow \infty} I_L(\xi_{n_k}) = \inf\{\, I_L(\eta) \; : \; \eta \in D_0 \,\}, 
    \end{align*}
    and so $\xi'$ is a minimizer so long as it belongs to $D_0$.  That is, we must show $\xi' \in S_0$ and that $\gamma^{\xi'}$ welds $\mathcal{P}$.  The first property is immediate since $I_L(\xi)< \infty$; in fact, $\gamma^{\xi'}$ is a $K$-quasiarc for some $K = K(I_L(\xi'))$ \cite[Prop. 2.1]{WangReverse}. We claim $\gamma^{\xi'}$ welds $\mathcal{P}$ by Theorem \ref{Thm:DriverToWeldingConvergence}.  Indeed, by extending $\xi'$ past $T_U$ to $[0,T_U']$ with the constant value $\xi'(T_U)$ on $[T_U,T_U']$, if necessary, we may assume that $[x_N, y_N]$ is in the interior of the interval welded by $\xi'$.  Similarly extending all the $\xi_{n_k}$ on $[T_U,T_U']$ to be constantly their terminal value $\xi_{n_k}(T_U)$, we still have uniform convergence on $[0,T_U']$, and Theorem \ref{Thm:DriverToWeldingConvergence} then yields
    \begin{align*}
        y_j = \varphi_{n_k}(x_j) \rightarrow \tilde{\varphi}(x_j)
    \end{align*}
    for each $j$, where $\varphi_{n_k}$ and $\varphi'$ are the weldings associated to $\xi_{n_k}$ and $\xi'$, respectively.  Thus $\xi' \in D_0$ and minimizers $\gamma_{\mathcal{P}} = \gamma^{\xi'}$ exist among all curves welding $\mathcal{P}$.
    
    \bigskip
    \noindent $(ii)$ Now suppose we have a sequence of partitions $\mathcal{P}_n = \{(x_j,y_j) \}_{j=0}^{N_n}$ with $|\mathcal{P}_n| \rightarrow 0$, and let $\gamma_n$ a minimizer for welding $\mathcal{P}_n$ with upwards driver $\xi_n \in S_0([0,T_n])$, which welds $-a$ to $b$ at time $T_n$.  By extending the drivers by the ending values, we may again assume there is a single interval $[0,T_U]$ on which all the $\xi_n$ are defined, and furthermore that we have $\epsilon>0$ such that each $\xi_n$ welds an interval including \begin{align}\label{Int:bigger}
        [-a-\epsilon,b+\epsilon]
    \end{align}
    by time $T_U$.  
    
    We show $\xi_n \rightarrow \xi$ on $[0,T_U]$ (we have also extended $\xi$ by constantly $\xi(T)$, if needed).  First note that the weldings $\varphi_n$ for $\xi_n$ converge uniformly to the welding $\varphi$ for $\xi$, which is clear from monotonicity and the agreement on $\mathcal{P}_n$ with $|\mathcal{P}_n| \rightarrow 0$.  As above, uniformly-bounded energy implies $\{\xi_n\}$ is precompact, and taking any uniform limit $\xi_{n_k} \rightarrow \xi'$ on $[0,T_U]$ we have that
    \begin{align*}
        I_L(\xi') \leq \liminf_{k \rightarrow \infty}I_L(\xi_{n_k}) \leq I_L(\xi),
    \end{align*}
    and so $\xi'$ has finite energy and generates a simple curve $\gamma'$, as noted above.\footnote{Of course, ``prime'' here is notation and has nothing to do with derivative.}  In particular, we find 
    \begin{align*}
        \max\{\,\tau(-a-\epsilon/2;\xi'), \tau(b+\epsilon/2;\xi')\,\} < T_U
    \end{align*}
    by \eqref{Int:bigger} and Theorem \ref{Cor:TimesPointwise}, and so by Theorem \ref{Thm:DriverToWeldingConvergence}, $\xi'$ welds identically to $\xi$ on $[-a,b]$.  By injectivity of the welding-to-curve map in the category of quasiarcs, the curves $\xi'$ and $\xi$ generate by welding $[-a,0]$ to $[0,b]$ are the same, up to post-composition by affine map $z \mapsto cz+d$ for some $c,d \in \mathbb{R}$.  However, since both $\gamma'$ and $\gamma$ are normalized by the Loewner flow, we have $c=1,d=0$ and consequently that $\xi' \equiv \xi$ on $[0,T_U]$.  
    
    All subsequential limits of $\{\xi_n\}$ are therefore $\xi$, and we conclude $\xi_n \xrightarrow{u} \xi$ on $[0,T_U]$.  In particular, convergence of hitting times yields
    \begin{align}\label{Conv:ConstructedTimesToTrue}
        T_n = \tau(-a; \xi_n) \rightarrow \tau(-a; \xi) = T
    \end{align}
    and hence, by uniformity also yields
    \begin{align}\label{Conv:CurveBases}
        \xi_n(T_n) \rightarrow \xi(T).
    \end{align}
    
    It remains to show that the curves $\gamma_n$ generated by $\xi_n$ converge to $\gamma$ on any $[0,T']$. Recall that since the inverse Loewner transform $\eta \mapsto \gamma^{\eta}$ is not continuous from $S_0([0,T_U])$ to $C([0,T_U])$ \cite[Ex. 4.49]{Lawler}, this is not immediate.  We start by rescaling the minimizers to all the have the same capacity time $T$, setting $\alpha_n := \sqrt{T/T_n}$ and $\tilde{\gamma}_n := \alpha_n\gamma_n$, and show that the $\tilde{\gamma}_n$ converge uniformly in their half-plane capacity parametrizations to $\gamma$ on $[0,T]$.  Note that the $\tilde{\gamma}_n$ have driving functions $\tilde{\xi}_n(\cdot) = \alpha_n \xi(\cdot/\alpha_n^2)$ which satisfy
    \begin{align*}
        |\tilde{\xi}_n(t)-\xi(t)| \leq& \sqrt{\frac{T}{T_n}} \Big| \xi_n\Big( t \, \frac{T_n}{T} \Big) - \xi\Big( t \, \frac{T_n}{T} \Big) \Big| + \Big| \sqrt{\frac{T}{T_n}}\xi\Big( t \, \frac{T_n}{T} \Big) - \xi(t)\Big|,
    \end{align*}
    where we are using the extension of $\xi$ to $[0,T_U]$ as needed.  For large $n$, this is small by \eqref{Conv:ConstructedTimesToTrue}, the uniform continuity of $\xi$ on $[0,T_U]$, and the convergence $\xi_n \xrightarrow{u} \xi$, showing
    \begin{align}\label{Conv:ScaledDrivers}
        \|\tilde{\xi}_n - \xi\|_{\infty[0,T]} \rightarrow 0,
    \end{align}
    which we will use this to prove 
    \begin{align}\label{Conv:GoalCurveConvergence}
        \|\tilde{\gamma}_n - \gamma\|_{\infty[0,T]} \rightarrow 0.
    \end{align}
    
    To begin with, note that $\{ \tilde{\gamma}_n \}$ is a bounded sequence both in diameter and Loewner energy (the former by \cite[top of p.74]{Lawler} again).  By Proposition \ref{Prop:Holder1/2}, the latter implies it is also equicontinuous in capacity parametrization, and thus a precompact family.  Taking any uniform subsequential limit $\tilde{\gamma}_{n_k} \rightarrow \tilde{\gamma}$ on $[0,T]$, we show that $\tilde{\gamma}$ is the half-plane capacity parametrization for $\gamma$, which is to say,
    \begin{align}\label{Eq:GoalSubsequentialCurveConvergence}
        \tilde{\gamma}([0,t])=\gamma([0,t])
    \end{align}
    for each $0 \leq t \leq T$.  As we are considering an arbitrary subsequential limit $\tilde{\gamma}$, if \eqref{Eq:GoalSubsequentialCurveConvergence} holds, so does \eqref{Conv:GoalCurveConvergence}.
    
    We first show \eqref{Eq:GoalSubsequentialCurveConvergence} for $t=T$.  Since the $\tilde{\gamma}_{n}$ are uniformly $K$-quasiarcs, we can write $\tilde{\gamma}_{n} = q_{n}([0,i])$ for some $K$-quasiconformal self-map $q_{n}$ of $\mathbb{H}$ that fixes $\infty$, sends $0$ to the base of curve $\tilde{\xi}_n(T)$, and $i$ to its tip $\tilde{\gamma}_{n}(T)$ \cite[Prop. 2.1]{WangReverse}.  We claim that
    \begin{align}\label{Ineq:PointsSeparated}
        |q_{n}(i) - q_{n}(0)| = |\tilde{\gamma}_n(T) - \tilde{\gamma}_n(0)|\geq \epsilon>0
    \end{align}
    for all $n$.  If so, then $\{q_n\}$ is a normal family, as then the points $\{0,i,\infty\}$ have images under the $q_n$ which are uniformly bounded below in spherical distance (we extend the $q_n$ by reflection $q_n(\bar{z}) = \overline{q_n(z)}$ to obtain quasiconformal mappings of the sphere $\hat{\mathbb{C}}$ and apply \cite[Thm. 2.1]{LehtoQC}).  Recall that if $\tilde{\gamma}_n \subset \overline{B_{\tilde{\gamma}_n(0)}(R)}\cap \mathbb{H}$, then $2T = \hcap(\tilde{\gamma}_n) \leq R^2$, and so there exists a point $\tilde{\gamma}_n(t_n)$ on $\tilde{\gamma}_n$ such that 
    \begin{align}\label{Ineq:PtFarAway}
        \sqrt{T} \leq |\tilde{\gamma}_n(t_n) - \tilde{\gamma}_n(0)|,
    \end{align}
    say.  As the $\tilde{\gamma}_n$ are uniformly $K$-quasiarcs, by Ahlfors' three point condition we have some $C=C(K)$ such that
    \begin{align*}
       |\tilde{\gamma}_n(0) -\tilde{\gamma}_n(t_n)| \leq C|\tilde{\gamma}_n(0) - \tilde{\gamma}_n(T)|.
    \end{align*}
    Combined with \eqref{Ineq:PtFarAway}, we therefore have \eqref{Ineq:PointsSeparated} and consequently normality of $\{q_n\}$.
    
    In particular, for $\tilde{\gamma}_{n_k} = q_{n_k}([0,i])$, we may move to a further subsequence of $\{q_{n_k}\}$, which we also label the same, and obtain a locally-uniform limit $\tilde{q}$, which is either a point in $\mathbb{R} \cup \{\infty\}$ or a $K$-quasiconformal homeomorphism $\mathbb{H}$ \cite[Thm. 2.3]{LehtoQC}.  To see that the former cannot happen, recall by \cite[Prop. 3.1]{WangReverse} that $I_L(\tilde{\gamma}_n) \geq -8\log(\sin(\theta_n))$, where $\theta_n = \arg(\tilde{\gamma}_n(t_n) - \tilde{\gamma}_n(0))$, and so $\theta_n$ is uniformly bounded away from 0 and $\pi$.  Combined with \eqref{Ineq:PtFarAway} and boundedness of $\{\tilde{\gamma}_n\}$, we conclude that the $q_{n_k}$ cannot degenerate to any $x \in \mathbb{R}$ or blow up to $\infty$, and hence that the $\tilde{\gamma}_{n_{k}}$ have a limiting Jordan curve $\tilde{\Gamma}:=q([0,i])$, where the limit is uniform in the parametrization given by the quasiconformal maps.  
    
    In particular, $\hcap(\tilde{\Gamma}) = 2T$ by Corollary \ref{Cor:hcapContinuousCurves}.  For any $0 \leq t \leq T$, let $y_t \in [0,1]$ and $t_{n_k} \in [0,T]$ be such that $q([0,iy_t]) = \tilde{\Gamma}([0,t])$ and $\tilde{\gamma}_{n_k}(t_{n_k})=q_{n_k}(iy_t)$.  We then have the point-wise convergence
    \begin{align*}
        \tilde{\gamma}_{n_k}(t_{n_k}) \rightarrow \tilde{\Gamma}(t),
    \end{align*}
    and since $q_{n_k} \xrightarrow{u} q$ on $[0,iy_t]$, by Corollary \ref{Cor:hcapContinuousCurves} again we conclude $t_{n_k} \rightarrow t$, and thus that
    \begin{align*}
        \tilde{\gamma}_{n_k}(t) \rightarrow \tilde{\Gamma}(t)
    \end{align*}
    by Proposition \ref{Prop:Holder1/2}.  Hence our subsequential limit $\tilde{\gamma} = \tilde{\Gamma}$ is a simple curve.
    
    By \eqref{Conv:ScaledDrivers} we have that the downwards drivers $\tilde{\lambda}_{n_k}(t) = \tilde{\xi}_{n_k}(T-t) - \tilde{\xi}_{n_k}(T)$ for the centered curves $\tilde{\gamma}_{n_k} - \tilde{\xi}_{n_k}(0)$ converge uniformly on $[0,T]$ to $\lambda(t) = \xi(T-t) - \xi(T)$, the downwards driver for $\gamma-\xi(T)$.  Proposition \ref{Prop:LMR} then yields $\tilde{\gamma} = \gamma$, showing that all subsequential limits are the same, and completing the proof of \eqref{Conv:GoalCurveConvergence}.
    
    Note \eqref{Lim:CurvesConverge} is an immediate consequence: $\gamma_n$ is defined on $[0,T'] \subset [0,T)$ for large $n$ by \eqref{Conv:ConstructedTimesToTrue}, and for $0 \leq t \leq T'$ we observe
    \begin{align*}
        |\gamma_n(t) - \gamma(t)| &\leq \Big|\gamma_n(t) - \sqrt{\frac{T}{T_n}}\gamma_n\Big(t \frac{T_n}{T} \Big)  \Big| + \|\tilde{\gamma}_n-\gamma\|_{\infty[0,T]}\\
        &\leq B\sqrt{|T-T_n|} + \|\tilde{\gamma}_n-\gamma\|_{\infty[0,T]},
    \end{align*}
    where the second line follows from \eqref{Conv:ConstructedTimesToTrue}, uniform boundedness of $\{\gamma_n\}$, and Proposition \ref{Prop:Holder1/2}.  We conclude \eqref{Lim:CurvesConverge} holds.
    
    We lastly turn to the energy limit \eqref{Lim:EnergiesIncrease}.  By minimization,
    \begin{align*}
        I_L(\gamma_n) = \frac{1}{2}\int_0^{T_n} \dot{\xi}_n^2(t) dt \leq \frac{1}{2}\int_0^{T} \dot{\xi}^2(t) dt = I_L(\gamma)
    \end{align*}
    for all $n$, and so $\limsup_{n \rightarrow \infty} I_L(\gamma_n) \leq I_L(\gamma)$.  Write $I_{L,t}(\cdot)$ for the Loewner energy on an interval $[0,t]$.  Since our extended drivers satisfy $\xi_n \xrightarrow{u} \xi$ on $[0,T_U]$, by energy lower semicontinuity \cite[\S2.2]{WangReverse},
    \begin{align*}
        I_{L,T}(\gamma) = I_{L,T_U}(\xi) \leq \liminf_{n \rightarrow \infty} I_{L, T_U}(\xi_n) = \liminf_{n \rightarrow \infty} I_{L, T_n}(\xi_n) \leq I_{L,T}(\gamma),
    \end{align*}
    where we recall that extending the drivers did not add any energy.  Thus $\limsup I_L(\gamma_n) \leq \liminf I_L(\gamma_n)$ and the claimed limit follows.  If the partitions are nested, $\gamma_{n+1}$ also welds $\mathcal{P}_n$, showing $I_L(\gamma_n) \leq I_L(\gamma_{n+1})$.
\end{proof}

\section{Problems}\label{Sec:Problems}
We close with two problems that appear natural from our study of drivers, weldings, and hitting times.

\begin{problem}
      Suppose weldings $\varphi_n, \varphi$ correspond to drivers $\xi_n, \xi \in S_0$.  If all the weldings share a fixed modulus of continuity $\omega$, does $\varphi_n \xrightarrow{u} \varphi$ imply $\xi_n \xrightarrow{u} \xi$?
\end{problem}

\begin{problem}
      Determine if $\xi \mapsto \tau$ is injective on a larger subcollection $S' \subset S$ (with equality, perhaps), beyond just constant drivers.  Is the map a homeomorphism onto its image of $S'$?
\end{problem}


\section{Appendix: Integral formulas relating $\xi, \tau$ and $\varphi$}\label{Appendix:Formulas}

\begin{lemma}
Let $\xi$ be an upwards driver generating a simple curve, and let $x_0$ be a point with $\tau(x_0;\xi)<\infty$.  The hitting time $\tau := \tau(x_0;\xi)$ satisfies
\begin{align}\label{Eq:HittingTimeIntegral1}
    \tau = \frac{1}{4}(x_0^2-x(\tau)^2)-\int_0^{\tau}\frac{\xi(t)}{x(t)-\xi(t)}dt.
\end{align}
Furthermore, if $y_0 = \varphi(x_0)$ is the image of $x_0$ under the conformal welding $\varphi$ generated by $\xi$, then 
\begin{align}\label{Eq:HittingTimeIntegral2}
    \varphi(x_0)^2 - x_0^2 = 4\int_0^\tau \frac{\xi(t)(y(t)-x(t))}{(\xi(t)-x(t))(y(t)-\xi(t))}dt.
\end{align}
\end{lemma}

\begin{proof}
    If $x_0 = \xi(0)$ the first formula is obvious.  If not, we observe that $\partial_t (x(t)^2) = -4x(t)/(x(t)-\xi(t))$, yielding
    \begin{align*}
        x(\tau-\epsilon)^2 - x_0^2 = -4 \int_0^{\tau-\epsilon} \frac{x(t)}{x(t)-\xi(t)}dt.
    \end{align*}
    Since $t \mapsto x(t)$ is continuous up to $t=\tau$, sending $\epsilon \rightarrow 0$ gives
    \begin{equation*}
        x(\tau)^2 - x_0^2 = -4 \int_0^{\tau} \frac{x(t)}{x(t)-\xi(t)}dt = -4 \tau - 4 \int_0^\tau \frac{\xi(t)}{x(t)-\xi(t)}dt,
    \end{equation*}
    which yields \eqref{Eq:HittingTimeIntegral1}.
    
    Secondly, since $x(\tau) = y(\tau) = \xi(\tau)$, using \eqref{Eq:HittingTimeIntegral1} for both $x_0$ and $y_0$ and equating the right-hand sides yields
    \begin{align*}
        x_0^2 - 4\int_0^\tau \frac{\xi(t)}{x(t)-\xi(t)}dt =  y_0^2 - 4\int_0^\tau \frac{\xi(t)}{y(t)-\xi(t)}dt,
    \end{align*}
    which is equivalent to \eqref{Eq:HittingTimeIntegral2}.
\end{proof}

\bibliographystyle{plain}
\bibliography{bib}

\end{document}